\documentclass[a4paper,12pt]{article}
\usepackage[latin1]{inputenc}
\usepackage{amsthm}
\usepackage{amsfonts}
\usepackage[dvips]{graphicx}
\usepackage{fancyhdr}
\usepackage{latexsym}
\usepackage{amsmath}
\usepackage{color}

\newcommand{\e}{\epsilon}
\newcommand{\phie}{\varphi_{\epsilon}}
\newcommand{\ue}{u_{\epsilon}}
\newcommand{\mue}{\mu_{\epsilon}}
\newcommand{\Fe}{F_{\epsilon}}
\newcommand{\Fie}{F_{1\epsilon}}
\newcommand{\UU}{\color{blue}}
\newcommand{\EE}{\color{black}}

\newtheorem{lem}{Lemma}
\newtheorem{thm}{Theorem}
\newtheorem{defn}{Definition}
\newtheorem{prop}{Proposition}
\newtheorem{cor}{Corollary}
\newtheorem{oss}{Remark}

\allowdisplaybreaks[4]

\begin{document}

\title{\large\textbf{NONLOCAL CAHN-HILLIARD-NAVIER-STOKES
SYSTEMS WITH SINGULAR POTENTIALS}}

\author{
{\sc Sergio Frigeri}\\
Dipartimento di Matematica {\it F. Enriques}
\\Universit\`{a} degli Studi di Milano\\Milano I-20133, Italy\\
\textit{sergio.frigeri@unimi.it}
\\
\\
{\sc Maurizio Grasselli}\\
Dipartimento di Matematica {\it F. Brioschi}\\
Politecnico di Milano\\
Milano I-20133, Italy \\
\textit{maurizio.grasselli@polimi.it}}

\maketitle

%%%%%%%%%%%%%%%%%%%%%%%%%%%%%%%%%%%%%%%%%%%%%%%%%%%%%%%%%%%%%%%%%%%%%%%%%%%%%%%%%%%%%%%%%%%%%%%%%%%%%%%%%%%%%%%%%%%%%%%%%%%
\begin{abstract}\noindent
Here we consider a Cahn-Hilliard-Navier-Stokes system characterized
by a nonlocal Cahn-Hilliard equation with a singular (e.g., logarithmic)
potential. This system originates from a diffuse interface model for
incompressible isothermal mixtures of two immiscible fluids. We have
already analyzed the case of smooth potentials with arbitrary
polynomial growth. Here, taking advantage of the previous results, we
study this more challenging (and physically relevant) case. We first
establish the existence of a global weak solution with no-slip and
no-flux boundary conditions. Then we prove the existence of the global
attractor for the 2D generalized semiflow (in the sense of J.M.~Ball).
We recall that uniqueness is still an open issue even in 2D. We also
obtain, as byproduct, the existence of a connected global attractor for
the (convective) nonlocal Cahn-Hilliard equation. Finally, in the 3D
case, we establish the existence of a trajectory attractor (in the sense
of V.V.~Chepyzhov and M.I.~Vishik).
\end{abstract}

\section{Introduction}
\setcounter{equation}{0} In \cite{CFG} we have introduced and
analyzed an evolution system which consists of the Navier-Stokes
equations for the fluid velocity $u$ suitably coupled with a non-local
convective Cahn-Hilliard equation for the order parameter $\varphi$
on a given (smooth) bounded domain
$\Omega\subset\mathbb{R}^d$, $d=2,3$. This system derives from
a diffuse interface model which describes the evolution of an
incompressible mixture of two immiscible fluids (see, e.g.,
\cite{GPV,HH,HMR,JV,M} and references therein). We suppose that
the temperature variations are negligible and the density is constant
and equal to one. Thus $u$ represents an average velocity and
$\varphi$ the relative concentration of one fluid (or the difference of
the two concentrations). Then the nonlocal
Cahn-Hilliard-Navier-Stokes system reads as follows
\begin{align}
&\varphi_t+u\cdot\nabla\varphi=\Delta\mu,\label{eq1}\\
&u_t-\mbox{div}(2\nu(\varphi)Du)+(u\cdot\nabla)u+\nabla\pi
=\mu\nabla\varphi+h,\label{eq2}\\
&\mu=a\varphi-J\ast\varphi+F'(\varphi),\label{eq3}\\
&\mbox{div}(u)=0,\label{eq4}
\end{align}
in $\Omega\times (0,+\infty)$. We endow the system with the
boundary and initial conditions
\begin{align}
&\frac{\partial\mu}{\partial n}=0,\quad u=0,
\quad\mbox{on }\partial\Omega,\label{eq5}\\
&u(0)=u_0,\quad\varphi(0)=\varphi_0,
 \quad\mbox{in }\Omega,\label{eq6}
\end{align}
where $n$ is the unit outward normal to $\partial\Omega$. Here
$\nu$ is the viscosity, $\pi$ the pressure, $h$ denotes an external
force acting on the fluid mixture, $J:\mathbb{R}^d \to \mathbb{R}$
is a suitable interaction kernel, $a$ is a coefficient depending on $J$
(see section below for the related assumptions), $F$ is the
configuration potential which accounts for the presence of two phases.

Here we prove the existence of a global weak solution when the
double-well potential $F$ is assumed to be singular in $(-1,1)$, that
is, its derivative is unbounded at the endpoints. A typical situation of
physical interest is the following (see \cite{CH})
\begin{align}
&F(s) = \frac{\theta}{2}((1+s)\log(1+s)+(1-s)\log(1-s)) -\frac{\theta_c}{2}s^2,
\label{log}
\end{align}
where $\theta$, $\theta_c$ are the (absolute) temperature and the critical temperature, respectively.
If $0<\theta< \theta_c$ then phase separation occurs, otherwise the mixed phase is stable.
We recall that the logarithmic terms are related to the entropy of the system.

For the existence of a weak solution, we take advantage of our
previous analysis for regular potentials (i.e., defined on the whole
$\mathbb{R}$) with polynomially controlled growth of arbitrary order
(see \cite{CFG}) and we use a suitable approximation procedure
inspired by \cite{EG}. Then, we extend to potentials like \eqref{log}
the results obtained in \cite{FG} for regular potentials. Such results
are concerned with the global longtime behavior of (weak) solutions.
More precisely, in the spirit of \cite{Ba}, we can define a generalized
semiflow in 2D and prove that it possesses a global (strong) attractor
by using the energy identity. Then we analyze the 3D case by means of
the trajectory approach introduced in \cite{Se} and generalized in
\cite{CV,CV2}. In this framework, we show the existence of a
trajectory attractor.

We recall that the chemical potential of the corresponding local
Cahn-Hilliard-Navier-Stokes system is given by $\mu = -\Delta\varphi
+ F'(\varphi)$. Therefore it can be seen as an approximation of the
nonlocal one (cf. \cite{CFG} and references therein). The local system
with a singular potential has been analyzed in \cite{A1,A2,B} (for
regular potentials see, e.g., \cite{GG1,GG2,S,ZWH} and references
therein). Most of the results known for the Navier-Stokes equations
essentially hold for the coupled (local) system as well. On the
contrary, in the nonlocal case, due to the weaker smoothness of
$\varphi$, proving uniqueness and/or getting higher-order estimates
seem a non-trivial task even in dimension two (see \cite{CFG,FG}).

We conclude by observing that the technique we use in 2D can be
easily adapted to show that the (convective) Cahn-Hilliard equation
with a singular potential has a connected global (strong) attractor (for
regular potentials see \cite{FG} and references therein, cf. also
\cite{AW,DD} for results on the local case).

The plan goes as follows. In the next section, we introduce the weak
formulation of our problem. Then we state the existence
theorem whose proof is given in Section 3.
Section~4 is devoted to the global attractor in 2D, while Section~5 is concerned with the
existence of the trajectory attractor.

\section{Weak solutions and existence theorem}
\setcounter{equation}{0} Let us set $H:=L^2(\Omega)$ and
$V:=H^1(\Omega)$. For every $f\in V'$ we denote by $\overline{f}$
the average of $f$ over $\Omega$, i.e.,
$$\overline{f}:=\frac{1}{|\Omega|}\langle f,1\rangle.$$
Here $|\Omega|$ stands for the Lebesgue measure of
$\Omega$.

Then we introduce the spaces
$$V_0:=\{v\in V:\overline{v}=0\},\qquad V_0':=\{f\in V':\overline{f}=0\},$$
and the operator $A:V\to V'$, $A\in\mathcal{L}(V,V')$ defined by
$$\langle Au,v\rangle:=\int_{\Omega}\nabla u\cdot\nabla v\qquad\forall u,v\in V.$$
We recall that $A$ maps $V$ onto $V_0'$ and the restriction of $A$ to $V_0$ maps $V_0$ onto $V_0'$
isomorphically. Let us denote by $\mathcal{N}:V_0'\to V_0$ the inverse map defined by
$$A\mathcal{N}f=f,\quad\forall f\in V_0'\qquad\mbox{and}\qquad\mathcal{N}Au=u,\quad\forall u\in V_0.$$
As is well known, for every $f\in V_0'$, $\mathcal{N}f$ is the unique
solution with zero mean value of the Neumann problem
\begin{equation*}
\left\{\begin{array}{ll}
-\Delta u=f,\qquad\mbox{in }\Omega\\
\frac{\partial u}{\partial n}=0,\qquad\mbox{on }\partial\Omega.
\end{array}\right.
\end{equation*}
Furthermore, the following relations hold
\begin{align}
&\langle Au,\mathcal{N}f\rangle=\langle f,u\rangle,\qquad\forall u\in V,\quad\forall f\in V_0',\\
&\langle f,\mathcal{N}g\rangle=\langle g,\mathcal{N}f\rangle=\int_{\Omega}\nabla(\mathcal{N}f)
\cdot\nabla(\mathcal{N}g),\qquad\forall f,g\in V_0'.
\end{align}
We also consider the standard Hilbert spaces for the Navier-Stokes
equations (see, e.g., \cite{T})
$$
G_{div}:=\overline{\{u\in
C^\infty_0(\Omega)^d:\mbox{
div}(u)=0\}}^{L^2(\Omega)^d},\quad
V_{div}:=\{u\in H_0^1(\Omega)^d:\mbox{ div}(u)=0\}.
$$
We denote by $\|\cdot\|$ and $(\cdot,\cdot)$ the norm and the
scalar product on both $H$ and $G_{div}$, respectively. We recall
that $V_{div}$ is endowed with the scalar product
$$(u,v)_{V_{div}}=(\nabla u,\nabla v),\qquad\forall u,v\in V_{div}.$$
We shall also use the definition of the Stokes operator $S$
with no-slip boundary condition. More precisely, $S:D(S)\subset G_{div}\to G_{div}$ is defined as
$S:=-P\Delta$  with domain $D(S)=H^2(\Omega)^d\cap V_{div}$,
where $P:L^2(\Omega)^d\to G_{div}$ is the Leray projector. Notice that we have
$$(Su,v)=(u,v)_{V_{div}}=(\nabla u,\nabla v),\qquad\forall u\in D(S),\quad\forall v\in V_{div}$$
and $S^{-1}:G_{div}\to G_{div}$ is a self-adjoint compact operator in $G_{div}$.
Thus, according with classical spectral theorems,
it possesses a sequence $\{\lambda_j\}$ with $0<\lambda_1\leq\lambda_2\leq\cdots$ and $\lambda_j\to\infty$,
and a family $\{w_j\}\subset D(S)$ of eigenfunctions which is orthonormal in $G_{div}$.
It is also convenient to recall that the trilinear form $b$ which appears in the weak formulation of the
Navier-Stokes equations is defined as follows
$$b(u,v,w)=\int_{\Omega}(u\cdot\nabla)v\cdot w,\qquad\forall u,v,w\in V_{div}.$$

We suppose that the potential $F$ can be written in the following form
$$F=F_1+F_2,$$
where $F_1\in C^{(2+2q)}(-1,1)$, with $q$ a fixed positive integer,
and $F_2\in C^2([-1,1])$.

We can now list the assumptions on the kernel $J$, on the viscosity
$\nu$, on $F_1$, $F_2$ and on the forcing term $h$.

\begin{description}
\item[(A1)]$J\in W^{1,1}(\mathbb{R}^d),\quad
    J(x)=J(-x),\quad a(x) := \displaystyle
\int_{\Omega}J(x-y)dy \geq 0,\quad\mbox{a.e. } x\in\Omega$.
\item[(A2)]The function $\nu$ is locally Lipschitz on $\mathbb{R}$ and there exist $\nu_1,\nu_2>0$ such that
 $$\nu_1\leq\nu(s)\leq \nu_2,\qquad\forall s\in\mathbb{R}.$$
\item[(A3)]
There exist $c_1>0$ and $\e_0>0$ such that
\begin{align*}
&F_1^{(2+2q)}(s)\geq c_1,\qquad\forall s\in(-1,-1+\e_0]\cup[1-\e_0,1).
%\mbox{near }s=\pm 1.\\
\end{align*}

\item[(A4)] There exists $\e_0>0$ such that, for each
    $k=0,1,\cdots, 2+2q$ and each $j=0,1,\cdots, q$,
\begin{align*}
&F_1^{(k)}(s)\geq 0,\qquad\forall s\in[1-\e_0,1),\\
%\mbox{near }s=1,\\
&F_1^{(2j+2)}(s)\geq 0,\qquad F_1^{(2j+1)}(s)\leq 0,\qquad\forall s\in(-1,-1+\e_0].
%\mbox{near }s=-1.
\end{align*}

\item[(A5)] There exists $\e_0>0$ such that $F_1^{(2+2q)}$ is
    non-decreasing in $[1-\e_0,1)$
%near $s=1$
and non-increasing in $(-1,-1+\e_0]$.
%near $s=-1$.

\item[(A6)]There exist $\alpha,\beta\in\mathbb{R}$ with $\alpha+\beta>-\min_{[-1,1]}F_2''$ such that
\begin{equation*}
F_1^{''}(s)\geq\alpha,\qquad\forall s\in(-1,1),\quad
a(x)\geq \beta,\qquad\mbox{a.e. }x\in\Omega.
\end{equation*}

\item[(A7)]$\lim_{s\to\pm 1}F_1'(s)=\pm\infty.$

\item[(A8)] $h\in L^2(0,T;V_{div}')$ for all $T>0$.

\end{description}

\begin{oss}
{\upshape Assumptions (A3)-(A7) are satisfied in the case of the
physically relevant logarithmic double-well potential \eqref{log} for
any fixed positive integer $q$. In particular, setting
$$F_1(s)=\frac{\theta}{2}((1+s)\log(1+s)+(1-s)\log(1-s)),\qquad F_2(s)=-\frac{\theta_c}{2}s^2,$$
then (A6) is satisfied if and only if $\beta>\theta_c-\theta$.
}
\end{oss}

\begin{oss}
{\upshape The requirement $a(x)\geq\beta$ a.e $x\in\Omega$ is
crucial (see \cite[Rem.2.1]{BH1}, cf. also \cite{BH2}). For example,
in the case of the double-well smooth potential $F(s)=(s^2-1)^2$,
which is usually taken as a fairly good smooth approximation of the
singular potential, the existence result in \cite{CFG} requires the
condition $a(x)\geq\beta$ with $\beta>4$. }
\end{oss}

The notion of weak solution to problem \eqref{eq1}-\eqref{eq6} is given by
\begin{defn}
\label{wsdefn}
Let $u_0\in G_{div}$, $\varphi_0\in H$ with $F(\varphi_0)\in L^1(\Omega)$ and $0<T<+\infty$ be given.
A couple $[u,\varphi]$ is a weak solution to \eqref{eq1}-\eqref{eq6} on $[0,T]$ corresponding to $[u_0,\varphi_0]$
if
\begin{itemize}
\item  $u$, $\varphi$ and $\mu$ satisfy
\begin{align}
&u\in L^{\infty}(0,T;G_{div})\cap L^2(0,T;V_{div}),\label{re1}\\
&u_t\in L^{4/3}(0,T;V_{div}'),\qquad\mbox{if}\quad d=3,\label{re2}\\
&u_t\in L^2(0,T;V_{div}'),\qquad\mbox{if}\quad d=2,\label{re3}\\
&\varphi\in L^{\infty}(0,T;H)\cap L^2(0,T;V),\label{re4}\\
&\varphi_t\in L^2(0,T;V'),\label{re5}\\
&\mu=a\varphi-J\ast\varphi+F'(\varphi)\in L^2(0,T;V),\label{prmu}
\end{align}
and
\begin{align}
&\varphi\in L^{\infty}(Q),\qquad|\varphi(x,t)|<1\quad\mbox{a.e. }(x,t)\in Q:=\Omega\times(0,T);\label{re7}
\end{align}

\item for every $\psi\in V$, every $v\in V_{div}$ and for almost
any $t\in(0,T)$ we have
\begin{align}
&\langle\varphi_t,\psi\rangle+(\nabla\mu,\nabla\psi)=(u,\varphi\nabla\psi),\label{weakfor1}\\
&\langle u_t,v\rangle+(2\nu(\varphi)Du,Dv)+b(u,u,v)=-(\varphi\nabla\mu,v)+\langle h,v\rangle;\label{weakfor2}
\end{align}

\item the initial conditions $u(0)=u_0$, $\varphi(0)=\varphi_0$ hold.

\end{itemize}

\end{defn}

\begin{thm}
\label{existence} Assume that (A1)-(A8) are satisfied for some fixed
positive integer $q$. Let $u_0\in G_{div}$, $\varphi_0\in
L^{\infty}(\Omega)$ such that $F(\varphi_0)\in L^1(\Omega)$. In
addition, assume that $|\overline{\varphi_0}|<1$. Then, for every
$T>0$ there exists a weak solution $z:=[u,\varphi]$ to
\eqref{eq1}-\eqref{eq6} on $[0,T]$ corresponding to
$[u_0,\varphi_0]$ such that
$\overline{\varphi}(t)=\overline{\varphi_0}$ for all $t\geq [0,T]$
and
\begin{align}
%&u\in L^{\infty}(0,T;G_{div})\cap L^2(0,T;V_{div}),\label{re1}\\
%&u_t\in L^{4/3}(0,T;V_{div}'),\qquad\mbox{if}\quad d=3,\label{re2}\\
%&u_t\in L^2(0,T;V_{div}'),\qquad\mbox{if}\quad d=2,\label{re3}\\
&\varphi\in L^{\infty}(0,T;L^{2+2q}(\Omega)).
\label{phiLq}
%&\varphi\in L^{\infty}(0,T;L^{2+2q}(\Omega))\cap L^2(0,T;V),\label{re4}\\
%&\varphi_t\in L^2(0,T;V'),\label{re5}\\
%&\mu\in L^2(0,T;V),\label{prmu}
\end{align}
%and
%\begin{align}
%&\varphi\in L^{\infty}(Q),\qquad|\varphi(x,t)|<1\quad\mbox{a.e. }(x,t)\in Q,\label{re7}
%\end{align}
%where $Q:=\Omega\times(0,T)$.
Furthermore, setting
$$\mathcal{E}(u(t),\varphi(t))=\frac{1}{2}\|u(t)\|^2+\frac{1}{4}
\int_{\Omega}\int_{\Omega}J(x-y)(\varphi(x,t)-\varphi(y,t))^2 dxdy+\int_{\Omega}F(\varphi(t)),$$
the following energy
inequality holds
\begin{equation}
\mathcal{E}(u(t),\varphi(t)) +\int_s^t\Big(2\Vert\sqrt{\nu(\varphi)}Du(\tau)\|^2+\|\nabla\mu(\tau)\|^2\Big)d\tau
\leq\mathcal{E}(u(s),\varphi(s))  +\int_s^t\langle h(\tau),u(\tau)\rangle d\tau,\label{ei}
\end{equation}
for all $t\geq s$ and for a.a. $s\in (0,\infty)$, including $s=0$.
If $d=2$, the weak solution $z:=[u,\varphi]$ satisfies
\begin{equation}
\frac{d}{dt}\mathcal{E}(u,\varphi)+2\Vert\sqrt{\nu(\varphi)}Du\Vert^2+\|\nabla\mu\|^2=\langle h,u\rangle,
\label{idendiff}
\end{equation}
i.e., equality holds in \eqref{ei} for every $t\geq 0$.
\end{thm}

Recalling \cite[Corollary 1, Proposition 5]{FG}, we can also deduce an
existence (and uniqueness) result for the convective nonlocal
Cahn-Hilliard equation with a given velocity field.
\begin{cor}
\label{NLCH1} Assume that (A1) and (A3)-(A7) are satisfied for some
fixed positive integer $q$. Let $u\in
L^2_{loc}([0,\infty);V_{div}\cap L^\infty(\Omega)^d)$ be given
and let $\varphi_0\in L^{\infty}(\Omega)$ such that
$F(\varphi_0)\in L^1(\Omega)$. In addition, suppose that
$|\overline{\varphi_0}|<1$. Then, for every $T>0$, there exists a
unique $\varphi \in L^2(0,T;V)\cap H^1(0,T;V^\prime)$ which
fulfills \eqref{re7} and \eqref{phiLq},  solves \eqref{weakfor1} on
$[0,T]$ with $\mu$ given by \eqref{prmu} and initial condition
$\varphi(0)=\varphi_0$. In addition, for all $t\geq 0$, we have
$(\varphi(t),1)=(\varphi_0,1)$ and the following energy identity
holds
\begin{equation}
\frac{d}{dt}\left(\frac{1}{4}
\int_{\Omega}\int_{\Omega}J(x-y)(\varphi(x,t)-\varphi(y,t))^2 dxdy
+\int_{\Omega}F(\varphi(t))\right) +\|\nabla\mu\|^2 =
(u\varphi, \nabla\mu).
\label{energyCH}
\end{equation}
\end{cor}

\begin{oss}
{\upshape Note that, thanks to \eqref{re4}, \eqref{prmu} and
\eqref{ei}, we have that
$$F'(\varphi)\in L^2(0,T;V),\quad
F(\varphi)\in L^{\infty}(0,T;L^1(\Omega)),\qquad\forall T>0.$$
}
\end{oss}

\begin{oss}
{\upshape The regularity property \eqref{phiLq} does not follow from \eqref{re7}.
Indeed, recall that $L^{\infty}(0,T;L^{\infty}(\Omega))\subset L^{\infty}(Q)$
with strict inclusion.
}
\end{oss}

\section{Proof of Theorem~\ref{existence}}
\setcounter{equation}{0} We consider the following approximate
problem $P_{\epsilon}$: find a weak solution $[\ue,\phie]$ to
\begin{align}
&\phie'+\ue\cdot\nabla\phie=\Delta\mue,\label{Pe1}\\
&\ue'-\mbox{div}(\nu(\phie)2D\ue)+(\ue\cdot\nabla)\ue
+\nabla\pi_{\epsilon}
=\mue\nabla\phie+h,\label{Pe2}\\
&\mue=a\phie-J\ast\phie+\Fe'(\phie),\label{Pe3}\\
&\mbox{div}(\ue)=0,\\
&\frac{\partial\mue}{\partial n}=0,\quad\ue=0,
\quad\mbox{on }\partial\Omega,\\
&\ue(0)=u_0,\quad\phie(0)=\varphi_0,\quad\mbox{in }\Omega.
\end{align}
Problem $P_{\epsilon}$ is obtained from \eqref{eq1}-\eqref{eq6}
by replacing the singular potential $F$ with the smooth potential
$$\Fe=\Fie+\overline{F}_2,$$
where $\Fie$ is defined by
\begin{equation}
\Fie^{(2+2q)}(s)=\left\{\begin{array}{lll}
F_1^{(2+2q)}(1-\e),\qquad s\geq 1-\e\\
F_1^{(2+2q)}(s),\qquad|s|\leq 1-\e\\
F_1^{(2+2q)}(-1+\e),\qquad s\leq -1+\e
\end{array}\right.
\label{approxpot}
\end{equation}
and $\Fie(0)=F_1(0)$,
$\Fie'(0)=F_1'(0)$,$\dots$$\Fie^{(1+2q)}(0)=F_1^{(1+2q)}(0)$,
while $\overline{F}_2$ is a $C^2(\mathbb{R})$-extension of $F_2$
on $\mathbb{R}$ with polynomial growth satisfying
\begin{align}
\overline{F}_2(s)\geq\min_{[-1,1]}F_2-1,\qquad\overline{F}_2''(s)\geq\min_{[-1,1]}F_2'',
\qquad\forall s\in\mathbb{R}.
\label{F_2}
\end{align}

The following elementary lemmas are basics to obtain uniform (w.r.t.
$\e$) estimates for a weak solution to the approximate problem.

\begin{lem}
\label{Fe}
Suppose that (A3) and (A4) hold. Then, there exist $c_q,d_q>0$, which depend on $q$
but are independent of $\e$, and $\e_0>0$
such that
\begin{equation}
\Fe(s)\geq c_q|s|^{2+2q}-d_q,\qquad\forall s\in\mathbb{R},\quad\forall\e\in(0,\e_0].
\label{coerapproxpot}
\end{equation}
\end{lem}
\begin{proof}
By integrating \eqref{approxpot} we get
\begin{align}
&\Fie(s)=\left\{\begin{array}{lll}
\sum_{k=0}^{2+2q}\frac{1}{k!}F_1^{(k)}(1-\e)[s-(1-\e)]^k,\qquad s\geq 1-\e\\
F_1(s),\qquad|s|\leq 1-\e\\
\sum_{k=0}^{2+2q}\frac{1}{k!}F_1^{(k)}(-1+\e)[s-(-1+\e)]^k,\qquad s\leq -1+\e.
\end{array}\right.
\label{approxpot2}
\end{align}
Due to (A4) we have, for $\e$ small enough,
$$\Fie(s)\geq\frac{1}{(2+2q)!}F_1^{(2+2q)}(1-\e)
[s-(1-\e)]^{2+2q},\qquad\forall s\geq 1-\e,$$ so that, in particular,
$$\Fie(s)\geq\frac{1}{(2+2q)!}F_1^{(2+2q)}(1-\e)(s-1)^{2+2q},\qquad\forall s\geq 1,$$
and (A3) implies that (for $\e$ small enough)
$$\Fie(s)\geq 2c_q (s-1)^{2+2q}\geq c_q s^{2+2q}-d_q,\qquad\forall s\geq 1,$$
where $c_q=c_1/2(2+2q)!$ and $d_q$ is another constant depending
only on $q$. Furthermore, we have $\Fie(s)=F_1(s)\geq 0\geq c_q
s^{2+2q}-d_q$ for $0\leq s\leq 1-\e$, provided we choose $d_q\geq
c_q$, while for $1-\e\leq s\leq 1$ we have $\Fie\geq
2c_q[s-(1-\e)]^{2+2q}\geq 0\geq c_q s^{2+2q}-d_q$. Summing
up, we deduce that there exists $\e_0>0$ such that $\Fie(s)\geq c_q
s^{2+2q}-d_q$, for all $s\geq 0$ and for all $\e\in(0,\e_0]$. By
using \eqref{F_2} we also get \eqref{coerapproxpot} for $s\geq 0$.
Similarly we obtain \eqref{coerapproxpot} for $s\leq 0$.
\end{proof}

\begin{lem}
\label{Fe''} Suppose (A4) and (A6) hold. Then, setting
$c_0:=\alpha+\beta+\min_{[-1,1]}F_2''>0$, there exists $\e_1>0$
such that
\begin{equation}
\Fe''(s)+a(x)\geq c_0,\qquad\forall s\in\mathbb{R},\quad\mbox{a.e. }x\in\Omega,\quad\forall\e\in(0,\e_1].
\label{approxconvpert}
\end{equation}
\end{lem}

\begin{proof}
From \eqref{approxpot2} we have
\begin{align}
&\Fie''(s)=\left\{\begin{array}{lll}
\sum_{k=0}^{2q}\frac{1}{k!}F_1^{(k+2)}(1-\e)[s-(1-\e)]^k,\qquad s\geq 1-\e\\
F_1''(s),\qquad|s|\leq 1-\e\\
\sum_{k=0}^{2q}\frac{1}{k!}F_1^{(k)}(-1+\e)[s-(-1+\e)]^k,\qquad s\leq -1+\e.
\end{array}\right.
\end{align}
Assumption (A4) implies that for $\e$ small enough $\Fie''(s)\geq F_1''(1-\e)$ for $s\geq 1-\e$
and $\Fie''(s)\geq F_1''(-1+\e)$ for $s\leq -1+\e$.
Since $\Fie''(s)=F_1''(s)$ for
$|s|\leq 1-\e$, (A6) implies that there exists $\e_1>0$ such that
\begin{align}
&\Fie''(s)\geq\alpha,\qquad\forall s\in\mathbb{R},\quad\forall\e\in(0,\e_1].
\label{Fie''bdb}
\end{align}
This estimate together with \eqref{F_2} and (A6) imply \eqref{approxconvpert}.
\end{proof}

Due to the existence result proved in \cite{CFG}, for every $T>0$,
Problem $P_{\e}$ admits a weak solution $z_{\e}:=[\ue,\phie]$
such that
\begin{eqnarray}
& &\ue\in L^{\infty}(0,T;G_{div})\cap L^2(0,T;V_{div}),\label{df1}\\
& &\ue'\in L^{4/3}(0,T;V_{div}'),\qquad\mbox{if}\quad d=3,\\
& &\ue'\in L^2(0,T;V_{div}'),\qquad\mbox{if}\quad d=2,\\
& &\phie\in L^{\infty}(0,T;L^{2+2q}(\Omega))\cap L^2(0,T;V),\\
& &\phie'\in L^2(0,T;V'),\\
& &\mue\in L^2(0,T;V).
\end{eqnarray}

Indeed, it is immediate to check that all the assumptions of
\cite[Theorem 1]{CFG} and of \cite[Corollary 1]{CFG} are satisfied
for Problem $P_{\e}$. In particular, we use Lemma \ref{Fe}, Lemma
\ref{Fe''} and the fact that, due to the definition of $\Fie$ and to the
polynomial growth assumption on $\overline{F}_2$, assumption (H5)
of \cite[Theorem 1]{CFG} is trivially satisfied for each $\e>0$ (with
some constants depending on $\e$).

Furthermore, according to \cite[Theorem 1]{CFG} and using (A2), the approximate solution $z_{\e}:=[\ue,\phie]$
satisfies the following energy inequality
\begin{align}
&\frac{1}{2}\|\ue(t)\|^2+\frac{1}{4}\int_{\Omega}\int_{\Omega}J(x-y)(\phie(x,t)-\phie(y,t))^2dxdy+
\int_{\Omega}\Fe(\phie(t))\nonumber\\
&+\int_0^t(\nu_1\|\nabla\ue\|^2+\|\nabla\mue\|^2)d\tau\leq
\frac{1}{2}\|u_0\|^2+\frac{1}{4}\int_{\Omega}\int_{\Omega}J(x-y)(\varphi_0(x)-\varphi_0(y))^2dxdy\nonumber\\
&+\int_{\Omega}\Fe(\varphi_0)+\int_0^t\langle h,\ue\rangle d\tau,\qquad\forall t\in[0,T].
\label{eie}
\end{align}
From (A5) it is easy to see (cf. \eqref{dercont} and \eqref{dercont2}
below) that there exists $\e_1>0$ such that
\begin{align}
&\Fie(s)\leq F_1(s),\qquad\forall s\in(-1,1),\quad\forall\e\in(0,\e_1].
\label{Fecont}
\end{align}
 Therefore, using the assumptions on $\varphi_0$,
$u_0$ and Lemma 1, from \eqref{eie} we get the following estimates
\begin{align}
&\|\ue\|_{L^{\infty}(0,T;G_{div})\cap L^2(0,T;V_{div})}\leq c,\label{best1}\\
&\|\phie\|_{L^{\infty}(0,T;L^{2+2q}(\Omega))}\leq c,\label{best2}\\
&\|\nabla\mue\|_{L^2(0,T;H)}\leq c.\label{best3}
\end{align}
Henceforth $c$ will denote a positive constant which depends on the initial data, but is independent
of $\e$.

We then take the gradient of \eqref{Pe3} and multiply the resulting
identity by $\nabla\phie$ in $L^2(\Omega)$. Arguing as in
\cite{CFG}, we get
$$\|\nabla\mue\|^2\geq\frac{c_0^2}{4}\|\nabla\phie\|^2-k\Vert\phie\Vert^2,$$
with $k=2\Vert\nabla J\Vert_{L^1}^2$.  This last estimate together
with \eqref{best2} and \eqref{best3} yield
\begin{align}
&\|\phie\|_{L^2(0,T;V)}\leq c.
\label{best4}
\end{align}
As far as the bounds on the time derivatives $\{\ue'\}$ and
$\{\phie'\}$ are concerned, on account of \eqref{Pe1} and
\eqref{Pe2}, arguing by comparison as in \cite{CFG} one gets
\begin{align}
&\Vert\phie'\Vert_{L^2(0,T;V')}\leq c,
\label{dest1}\\
&\Vert\ue'\Vert_{L^2(0,T;V_{div}')}\leq c,\qquad d=2\label{dest2}\\
&\Vert\ue'\Vert_{L^{4/3}(0,T;V_{div}')}\leq c,\qquad d=3.\label{dest3}
\end{align}
In order to obtain an estimate for $\{\mue\}$ we need to control the
sequence of averages $\{\overline{\mue}\}$. To this aim observe
that equation \eqref{Pe1} can be written in abstract form as follows
\begin{align}
&\phie'+\ue\cdot\nabla\phie=-A\mue\qquad\mbox{in }V'.\label{Pe1abst}
\end{align}
Let us test \eqref{Pe1abst} by $\mathcal{N}(\Fe'(\phie)-\overline{\Fe'(\phie)})$ to get
\begin{align}
&\langle\Fe'(\phie)-\overline{\Fe'(\phie)},\mathcal{N}\phie'\rangle+\langle\mathcal{N}(\ue\cdot\nabla\phie),
\Fe'(\phie)-\overline{\Fe'(\phie)}\rangle\nonumber\\
&=-\langle\mue,\Fe'(\phie)-\overline{\Fe'(\phie)}\rangle \UU.\EE
\label{Pe1tested}
\end{align}
Recall that $\overline{\ue\cdot\nabla\phie}=0$. On the other hand,
we have
\begin{align}
&\langle\mue,\Fe'(\phie)-\overline{\Fe'(\phie)}\rangle=
\langle a\phie-J\ast\phie+\Fe'(\phie)-\overline{\Fe'(\phie)},\Fe'(\phie)-\overline{\Fe'(\phie)}\rangle
\nonumber\\
&\geq\frac{1}{2}\Vert\Fe'(\phie)-\overline{\Fe'(\phie)}\Vert^2-\frac{1}{2}\Vert a\phie-J\ast\phie\Vert^2
\geq\frac{1}{2}\Vert\Fe'(\phie)-\overline{\Fe'(\phie)}\Vert^2-C_J\Vert\phie\Vert^2.
\label{FeFe'}
\end{align}

Therefore, by means of \eqref{FeFe'} and \eqref{best2}, from
\eqref{Pe1tested} we deduce
\begin{align}
&\Vert\Fe'(\phie)-\overline{\Fe'(\phie)}\Vert\leq c(\Vert\mathcal{N}\phie'\Vert+\Vert\mathcal{N}(\ue\cdot\nabla\phie)\Vert+1)\nonumber\\
&\leq c(\Vert\phie'\Vert_{V_0'}+\Vert\ue\cdot\nabla\phie\Vert_{V_0'}+1).
\label{FeFe'2}
\end{align}

Observe now that, due to (A4) and (A5), there holds
\begin{align}
&|\Fie'(s)|\leq|F'_1(s)|,\qquad\forall s\in(-1,1),\quad
\forall\e\in(0,\e_1],
\label{dercont}
\end{align}
for some $\e_1>0$. Indeed, for $s\in[1-\e,1)$ we have
\begin{align}
&F_1'(s)=\sum_{k=0}^{2q}\frac{1}{k!}F_1^{(k+1)}(1-\e)[s-(1-\e)]^k
+\frac{1}{(2q+1)!}F_1^{(2q+2)}(\xi)[s-(1-\e)]^{1+2q}\nonumber\\
&\geq \sum_{k=0}^{1+2q}\frac{1}{k!}F_1^{(k+1)}(1-\e)[s-(1-\e)]^k=\Fie'(s),
\label{dercont2}
\end{align}
for $\e$ small enough, where $\xi\in(1-\e,s)$ and where we have
used the fact that, due to (A5), $F_1^{(2+2q)}(\xi)\geq
F_1^{(2+2q)}(1-\e)$. Arguing similarly, we get $\Fie'(s)\geq
F_1'(s)$ for $s\in(-1,-1+\e]$ and for $\e$ small enough. However,
due to (A4) and (A7), for $\e$ small enough we have that
$\Fie'(s)\geq F_1'(1-\e)\geq 0$ for $s\geq 1-\e$ and $\Fie'(s)\leq
F_1'(-1+\e)\leq 0$ for $s\leq -1+\e$. Recalling also that
$\Fie'(s)=F_1'(s)$ for $|s|\leq 1-\e$, we obtain \eqref{dercont}.

Let $s_0\in(-1,1)$ be such that $F'(s_0)=0$ (cf. (A7)) and introduce
\begin{align}
&H(s):=F(s)+\frac{a_{\infty}}{2}(s-s_0)^2,\qquad H_{\epsilon}(s):=\Fe(s)+\frac{a_{\infty}}{2}(s-s_0)^2,
\label{H}
\end{align}
for every $s\in(-1,1)$ and every $s\in\mathbb{R}$, respectively.
Observe that, owing to \eqref{approxconvpert}, $H'_{\epsilon}$ is
monotone and (for $\e$ small enough)
$H'_{\epsilon}(s_0)=F'(s_0)=0$.
%Now, observe that, due to (A4), the regular part $F_2$ of the singular potential $F$
%can be written as a quadratic perturbation of a convex function, i.e. $F_2$ can be
%represented as
%$$F_2(s)=G_2(s)-\frac{a_{\infty}}{2}s^2,$$
%with $G_2$ strictly convex and $a_{\infty}:=\|a\|_{\infty}$.
%Setting now $\widetilde{\Fie}:=\Fie+G_2$, observe that, due to \eqref{Fie''bdb}
%$$\widetilde{\Fie}''(s)=\Fie''(s)+G_2''(s)\geq m+c_2>0,\qquad\forall s\in\mathbb{R},\quad\forall\e\in(0,1).$$
%Hence, $\widetilde{\Fie}'$ is monotone
%and, due to the second assumption in (A3), $F_1'(0)+G_2'(0)=0$. Furthermore by (A3) we have $$|\Fie'(s)|\leq|F'_1(s)|,\qquad\forall s\in(-1,1),\quad\forall\e\in(0,\e_1],$$
Since $\overline{\varphi_0}\in(-1,1)$, we can apply an argument
devised by Kenmochi et al. \cite{KNP} (see also \cite{CGGS}) and
deduce the following estimate
\begin{align}
\delta\Vert H'_{\epsilon}(\phie)\Vert_{L^1(\Omega)}\leq
\int_{\Omega}(\phie-\overline{\varphi_0})(H'_{\epsilon}(\phie)-\overline{H'_{\epsilon}(\phie)})
+K(\overline{\varphi_0})
\label{Ken}
\end{align}
where $\delta$ depends on $\overline{\varphi_0}$ and
$K(\overline{\varphi_0})$ depends on $\overline{\varphi_0}$, $F$,
$|\Omega|$ and $a$. For the reader's convenience let us recall briefly
how \eqref{Ken} can be deduced. Fix $m_1,m_2\in(-1,1)$ such that
$m_1\leq s_0\leq m_2$ and $m_1<\overline{\varphi_0}<m_2$.
Introduce, for a.a. fixed $t\in(0,T)$, the sets
$$\Omega_0:=\{m_1\leq\phie(x,t)\leq m_2\},\quad
\Omega_1:=\{\phie(x,t)<m_1\},\quad\Omega_2:=\{\phie(x,t)>m_2\}.$$
Setting
$\delta:=\min\{\overline{\varphi_0}-m_1,m_2-\overline{\varphi_0}\}$
and
$\delta_1:=\max\{\overline{\varphi_0}-m_1,m_2-\overline{\varphi_0}\}$,
then for $\e$ small enough we have
\begin{align}
&\delta\Vert H'_{\epsilon}(\phie)\Vert_{L^1(\Omega)}=\delta\int_{\Omega_1}|H'_{\epsilon}(\phie)|
+\delta\int_{\Omega_2}|H'_{\epsilon}(\phie)|+\delta\int_{\Omega_0}|H'_{\epsilon}(\phie)|\nonumber\\
&\leq\int_{\Omega_1}(\phie(t)-\overline{\varphi_0})H'_{\epsilon}(\phie)+
\int_{\Omega_2}(\phie(t)-\overline{\varphi_0})H'_{\epsilon}(\phie)
+\delta\int_{\Omega_0}|H'_{\epsilon}(\phie)|\nonumber\\
&\leq\int_{\Omega}(\phie(t)-\overline{\varphi_0})H'_{\epsilon}(\phie)
+(\delta_1+\delta)\int_{\Omega_0}|H'_{\epsilon}(\phie)|\nonumber\\
&\leq\int_{\Omega}(\phie(t)-\overline{\varphi_0})H'_{\epsilon}(\phie)
+(\delta_1+\delta)\int_{\Omega_0}\Big\{|F_1'(\phie)|+|F_2'(\phie)|+a_{\infty}|\phie-s_0|\Big\},
\nonumber
\end{align}
where we have used \eqref{dercont}. We therefore get
\eqref{Ken} with $K(\overline{\varphi_0})$ given by
$$K(\overline{\varphi_0})=(\delta_1+\delta)|\Omega|\Big(\max_{[m_1,m_2]}(|F_1'|+|F_2'|)
+a_{\infty}\delta_2\Big),$$
with $\delta_2:=\max\{s_0-m_1,m_2-s_0\}$.
%\sup_{m_1\leq s\leq m_2}
On account of the definition of $H_{\epsilon}$ and recalling
\eqref{FeFe'2} we obtain
\begin{align}
&\Vert H'_{\epsilon}(\phie)-\overline{H'_{\epsilon}(\phie)}\Vert
\leq c(\Vert\phie'\Vert_{V_0'}+\Vert\ue\cdot\nabla\phie\Vert_{V_0'}+1)+a_{\infty}\Vert\phie-\overline{\varphi_0}
\Vert.
\label{tildeFe}
\end{align}

Therefore, by means of \eqref{Ken}-\eqref{tildeFe} and using the
following bound (cf. \eqref{best1} and \eqref{best2}, see
\cite{CFG} for details)
\begin{align*}
&\|\ue\cdot\nabla\phie\|_{L^2(0,T;V_0')}\leq c,
\end{align*}
we infer that there exists a function $L_{\overline{\varphi_0}}\in L^2(0,T)$ depending on $\overline{\varphi_0}$ such
that
\begin{align}
\Vert\Fe'(\phie)\Vert_{L^1(\Omega)}\leq L_{\overline{\varphi_0}}.
\label{Ken2}
\end{align}

Since $\int_{\Omega}\mue=\int_{\Omega}\Fe'(\phie)$, then
$\Vert\overline{\mue}\Vert_{L^2(0,T)}\leq c$. Hence by
Poincar\'{e}-Wirtinger inequality and \eqref{best3} we get
\begin{align}
\Vert\mue\Vert_{L^2(0,T;V)}\leq c.
\label{best5}
\end{align}

Estimates \eqref{best1}, \eqref{best2},
\eqref{best4}-\eqref{dest3}, \eqref{best5} and well-known
compactness results allow us to deduce that there exist functions
$u\in L^{\infty}(0,T;G_{div})\cap L^2(0,T;V_{div})$, $\varphi\in
L^{\infty}(0,T;L^{2+2q}(\Omega))\cap L^2(0,T;V)\cap
H^1(0,T;V')$, and $\mu\in L^2(0,T;V)$ such that, up to a
subsequence, we have
\begin{align}
& \ue\rightharpoonup u\qquad\mbox{weakly}^{\ast}\mbox{ in } L^{\infty}(0,T;G_{div}),
\quad\mbox{weakly in }L^2(0,T;V_{div}),\label{c1}\\
& \ue\to u\qquad\mbox{strongly in }L^2(0,T;G_{div}),\quad\mbox{a.e. in }\Omega\times(0,T),\label{c2}\\
& \ue'\rightharpoonup u_t\qquad\mbox{weakly in }L^{4/3}(0,T;V_{div}'),\quad d=3,\label{c3}\\
& \ue'\rightharpoonup u_t
\qquad\mbox{weakly in }L^2(0,T;V_{div}'),\quad d=2,\label{c4}\\
& \phie\rightharpoonup\varphi\qquad\mbox{weakly}^{\ast}\mbox{ in }L^{\infty}(0,T;L^{2+2q}(\Omega)),
\quad\mbox{weakly in }L^2(0,T;V),\label{c5}\\
& \phie\to\varphi\qquad\mbox{strongly in }L^2(0,T;H),\quad\mbox{a.e. in }\Omega\times(0,T),\label{c6}\\
& \phie'\rightharpoonup\varphi_t\qquad\mbox{weakly in }L^2(0,T;V'),\label{c7}\\
& \mue\rightharpoonup\mu\qquad\mbox{weakly in }L^2(0,T;V).\label{c8}
\end{align}

In order to pass to the limit in the variational  formulation for Problem
$P_{\e}$ and hence prove that $z=[u,\varphi]$ is a weak solution to
the original problem, we need to show that $|\varphi|<1$ a.e. in
$Q=\Omega\times(0,T)$. To this aim we adapt an argument devised
in \cite{DD}. Thus, for a.a. fixed $t\in(0,T)$, we introduce the sets
$$E_{1,\eta}^{\e}:=\{\phie(x,t)>1-\eta\},\quad E_{2,\eta}^{\e}
:=\{\phie(x,t)<-1+\eta\},$$
where $\eta\in(0,1)$ is chosen so that $s_0\in(-1+\eta,1-\eta)$
with $s_0$ such that $F'(s_0)=0$. For $\e$ small enough, recalling
that $H_{\e}'(s)\geq 0$ for $s\in[s_0,1)$ and $H_{\e}'(s)\leq 0$
for $s\in(-1,s_0]$, we can write
\begin{align}
&H_{\e}'(1-\eta)|E_{1,\eta}^{\e}|\leq\Vert H_{\e}'(\phie)\Vert_{L^1(\Omega)},
\qquad|H_{\e}'(-1+\eta)||E_{2,\eta}^{\e}|\leq\Vert H_{\e}'(\phie)\Vert_{L^1(\Omega)},
\label{DebDet1}
\end{align}
and observe that $\Vert H_{\e}'(\phie)\Vert_{L^1(\Omega)}\leq
L_{\overline{\varphi_0}}$ (cf. \eqref{Ken2}). Furthermore, as a
consequence of the pointwise convergence \eqref{c6} and by using
Fatou's lemma, it is easy to see that we have
\begin{align}
&|E_{1,\eta}|\leq\liminf_{\e\to 0}|E_{1,\eta}^{\e}|,\quad |E_{2,\eta}|\leq\liminf_{\e\to 0}|E_{2,\eta}^{\e}|,
\label{DebDet2}
\end{align}
where
$$E_{1,\eta}:=\{\varphi(x,t)>1-\eta\},\quad
E_{2,\eta}:=\{\varphi(x,t)<-1+\eta\}.$$
Hence, due to the pointwise convergence $H_{\e}'(s)\to H'(s)$,
for every $s\in(-1,1)$,
%where $\widetilde{F_1}:=F_1+G_2$,
we get from \eqref{DebDet1} and \eqref{DebDet2}
\begin{align}
&|E_{1,\eta}|\leq\frac{L_{\overline{\varphi_0}}}{H'(1-\eta)},
\qquad|E_{2,\eta}|\leq\frac{L_{\overline{\varphi_0}}}{|H'(-1+\eta)|}.
\end{align}
Letting $\eta\to 0$ and using (A7) we obtain
$|\{x\in\Omega:|\varphi(x,t)|\geq 1\}|=0$ for a.e. $t\in(0,T)$ and
therefore $|\varphi(x,t)|<1$ for a.e. $(x,t)\in Q$. This bound, the
pointwise convergence \eqref{c6} in $Q$ and the fact that $\Fe'\to
F'$ uniformly on every compact interval included in $(-1,1)$, entail
that
\begin{align}
&\Fe'(\phie)\to F'(\varphi)\quad\mbox{a.e. in }Q.
\label{pointFe'}
\end{align}

Convergences \eqref{c1}-\eqref{c8} and \eqref{pointFe'} allow us,
by a standard argument, to pass to the limit in the variational
formulation of Problem $P_{\e}$ and hence to prove that
$z=[u,\varphi]$ is a weak solution to \eqref{eq1}-\eqref{eq6}.

%In order to prove the energy inequality
%\eqref{ei},
Let us now establish the energy inequality \eqref{ei}. Let us first
show that \eqref{ei} holds for $s=0$ and $t>0$.
% let us observe that
Indeed, the energy inequality
%\eqref{eie}
satisfied by the approximate solution $z_{\epsilon}=[\ue,\phie]$ can be written in the form
\begin{align}
&\frac{1}{2}\Vert\ue(t)\Vert^2+
\frac{1}{2}\|\sqrt{a}\phie(t)\|^2-\frac{1}{2}(\phie(t),J\ast\phie(t))+\int_{\Omega}\Fe(\phie(t))\nonumber\\
&+\int_0^t\Big(2\Vert\sqrt{\nu(\phie)}D\ue\Vert^2+\|\nabla\mue\|^2\Big)d\tau\leq
\frac{1}{2}\|u_0\|^2+\frac{1}{2}\|\sqrt{a}\varphi_0\|^2-\frac{1}{2}(\varphi_0,J\ast\varphi_0)\nonumber\\
&+\int_{\Omega}\Fe(\varphi_0)+\int_0^t\langle h,\ue\rangle d\tau,\qquad\forall t>0.
\label{apprei}
\end{align}
We now use the strong convergences \eqref{c2} and \eqref{c6}, the
weak convergences \eqref{c1} and \eqref{c8}, the bound
\eqref{Fecont} for the approximate potential $\Fie$, the fact that
$\Fe(\phie(t))\to F(\varphi(t))$ a.e. in $\Omega$ and for a.e.
$t\in(0,T)$ (see \eqref{pointFe'}) and Fatou's lemma. Observe that,
as a consequence of the uniform bound
$\Vert\sqrt{\nu(\phie)}\Vert_{\infty}\leq\sqrt{\nu_2}$, of the
strong convergence $\sqrt{\nu(\phie)}\to\sqrt{\nu(\varphi)}$ in
$L^2(0,T;H)$ and of the weak convergence \eqref{c1}, we have
\begin{align}
&\sqrt{\nu(\phie)}D\ue\rightharpoonup\sqrt{\nu(\varphi)}Du,\qquad\mbox{weakly in }L^2(0,T;H).
\label{ddhh}
\end{align}
By letting $\e\to 0$, from \eqref{apprei} we infer that \eqref{ei}
holds for almost every $t>0$. Furthermore, due to the regularity
properties of the solution, there exists a representative
$z=[u,\varphi]$ such that $u\in C_w([0,\infty);G_{div})$ and
$\varphi\in C([0,\infty);H)$ (henceforth we shall always choose this
representative). Therefore, \eqref{ei} holds {\it for all} $t\geq 0$
since
%Then, it is easy to prove that
the function $\mathcal{E}(z(\cdot)):[0,\infty)\to\mathbb{R}$ is
lower semicontinuous. The lower semicontinuity of $\mathcal{E}$ is a
consequence of the fact that $F$ is a quadratic perturbation of a
(strictly) convex function in $(-1,1)$. Indeeed,  by (A6) we have that
$F''(s)\geq\alpha_{\ast}$, for all $s\in(-1,1)$, with
$\alpha_{\ast}=\alpha+\min_{[-1,1]}F_2''$. Then $F$ can be
written in the form
\begin{align}
&
F(s)=G(s)+\frac{\alpha_{\ast}}{2}s^2,
\label{Fqconvex}
\end{align}
with $G$ convex on $(-1,1)$ (see \cite[Lemma 2]{FG}).

Let us now prove that the energy inequality \eqref{ei} also holds
between two arbitrary times $s$ and $t$.
% Precisely we have
%\begin{align}
%&\mathcal{E}(z(t))+\int_s^t(\nu\|\nabla u\|^2+\|\nabla\mu\|^2)d\tau
%\leq\mathcal{E}(z(s))+\int_s^t\langle h, u\rangle d\tau,
%\label{eist}
%\end{align}
%for all $t\geq s$, for a.e. $s\in(0,\infty)$, including $s=0$.
Indeed, setting
\begin{align}
&\mathcal{E}_{\e}(z_{\e}(t))=\frac{1}{2}\Vert\ue(t)\Vert^2+
\frac{1}{2}\|\sqrt{a}\phie(t)\|^2-\frac{1}{2}(\phie(t),J\ast\phie(t))+\int_{\Omega}\Fe(\phie(t)),
\end{align}
and applying \cite[Lemma 3]{FG}, we deduce (see Remark
\ref{lem3FG}) that the approximate solution
$z_{\epsilon}=[\ue,\phie]$ satisfies
\begin{align}
&\mathcal{E}_{\e}(z_{\e}(t))+\int_s^t\Big(2\Vert\sqrt{\nu(\phie)}D\ue\Vert^2+\|\nabla\mue\|^2\Big)d\tau
\leq\mathcal{E}_{\e}(z_{\e}(s))+\int_s^t\langle h,\ue\rangle d\tau,
\label{eiepsilon}
\end{align}
for every $t\geq s$ and for a.e. $s\in(0,\infty)$, including $s=0$.

Define $G_{\e}$ in such a way that
\begin{align}
&
\Fe(s)=G_{\e}(s)+\frac{\alpha_{\ast}}{2}s^2,
\label{Fqconvexe}
\end{align}
with $\alpha_{\ast}$ as in \eqref{Fqconvex}. Since, due to
\eqref{Fie''bdb}, $G_{\e}$ is convex on $(-1,1)$,
%Since $\Fe$ is convex (actually $\Fe$ is a quadratic perturbation of a convex function
%but this is not restrictive),
then we can write
$$G_{\e}(\phie)\leq G_{\e}(\varphi)+G_{\e}'(\phie)(\phie-\varphi).$$
Hence, for every non-negative $\psi\in\mathcal{D}(0,t)$, we have
$$\int_{Q_t}G_{\e}(\phie)\psi\leq\int_{Q_t}G_{\e}(\varphi)\psi+\int_{Q_t}G_{\e}'(\phie)(\phie-\varphi)\psi,$$
where $Q_t:=\Omega\times(0,t)$. Thus, thanks to \eqref{best5} and
\eqref{c2}, we get
$$\Big|\int_{Q_t}G_{\e}'(\phie)(\phie-\varphi)\psi\Big|\leq c\Vert G_{\e}'(\phie)\Vert_{L^2(0,T;H)}
\Vert\phie-\varphi\Vert_{L^2(0,T;H)}\leq
c\Vert\phie-\varphi\Vert_{L^2(0,T;H)}\to 0,$$ as $\e\to 0$. Here
we have used the fact that, since
$\Vert\Fe'(\phie)\Vert_{L^2(0,T;H)}\leq c$ and
$G_{\e}'(\phie)=\Fe'(\phie)-\alpha_{\ast}\phie$, then $\Vert
G_{\e}'(\phie)\Vert_{L^2(0,T;H)}\leq c$. Therefore, by using
Lebesgue's theorem (recall \eqref{Fecont} and the fact that
$|\varphi|<1$ a.e. in $Q$) we find
$$\limsup_{\e\to 0}\int_{Q_t} G_{\e}(\phie)\psi
\leq\lim_{\e\to 0}\int_{Q_t} G_{\e}(\varphi)\psi
=\int_{Q_t}G(\varphi)\psi.$$ On the other hand, thanks to Fatou's
lemma and to the pointwise convergence $\Fe(\phie)\to F(\varphi)$,
we also have the liminf inequality. Then, on account of
\eqref{Fqconvex} and \eqref{Fqconvexe}, we deduce that
\begin{align}
&\int_{Q_t}\Fe(\phie)\psi\to\int_{Q_t}
F(\varphi)\psi,\qquad\forall\psi\in\mathcal{D}(0,t),\;\psi\geq 0.
\label{hx}
\end{align}
Let us multiply \eqref{eiepsilon} by a non-negative
$\psi\in\mathcal{D}(0,t)$ and integrate the resulting inequality
w.r.t. $s$ from $0$ and $t$, where $t>0$ is fixed. We obtain
\begin{align*}
&\mathcal{E}_{\e}(z_{\e}(t))\int_0^t\psi(s)ds+\int_0^t\psi(s)ds
\int_s^t\Big(2\Vert\sqrt{\nu(\phie)}D\ue\Vert^2
+\|\nabla\mue\|^2\Big)d\tau\nonumber\\
&\leq\int_0^t\mathcal{E}_{\e}(z_{\e}(s))\psi(s)ds
+\int_0^t\psi(s)ds\int_s^t\langle h,\ue\rangle d\tau.
\end{align*}
By using strong and weak convergences for the sequence
$\{z_{\e}\}$ and \eqref{hx}, passing to the limit as $\e\to 0$ in
the above inequality, we infer
\begin{align*}
&\mathcal{E}(z(t))\int_0^t\psi(s)ds+\int_0^t\psi(s)ds
\int_s^t\Big(2\Vert\sqrt{\nu(\varphi)}Du\Vert^2
+\|\nabla\mu\|^2\Big)d\tau\nonumber\\
&\leq\int_0^t\mathcal{E}(z(s))\psi(s)ds
+\int_0^t\psi(s)ds\int_s^t\langle h,u\rangle d\tau,
\end{align*}
which can be rewritten as follows
\begin{align*}
V_z(t)\int_0^t\psi(s)ds\leq \int_0^t V_z(s)\psi(s)ds,
\end{align*}
where
\begin{align*}
V_z(t):=\mathcal{E}(z(t))+\int_0^t\Big(2\Vert\sqrt{\nu(\varphi)}Du\Vert^2
+\|\nabla\mu\|^2\Big)d\tau-\int_0^t\langle h,u\rangle d\tau.
\end{align*}
Thus we have
\begin{align*}
\int_0^t (V_z(s)-V_z(t))\psi(s)ds\geq 0,\qquad
\forall\psi\in\mathcal{D}(0,t),\;\psi\geq 0,
\end{align*}
which implies that $V_z(t)\leq V_z(s)$ for a.e. $s\in (0,t)$.
Therefore, \eqref{ei} is proven.

Finally, for $d=2$, we can choose $u_t$ and $\varphi_t$ as test
functions in \eqref{weakfor1}-\eqref{weakfor2}, due to their
regularity properties, then use \eqref{Fqconvex} and
\cite[Proposition 4.2]{CKRS} and deduce \eqref{idendiff} (see
\cite{CFG} for details).

\begin{oss}
\label{lem3FG} {\upshape In \cite[Lemma 3]{FG} a growth
assumption is made on the regular potential (polynomial growth less
then 6 when $d=3$). Therefore, the application of \cite[Lemma
3]{FG} to obtain the approximate energy inequality \eqref{eiepsilon}
would require the condition $q=1$ (recall that the approximate
potential $\Fe$ has polynomial growth of order $2+2q$).
Nevertheless, by exploiting an argument of the same kind as above
and by suitably approximating regular potentials of arbitrary
polynomial growth by a sequence of potentials of polynomial growth
of order less then 6, it is not difficult to improve \cite[Lemma 3]{FG}
and remove such growth assumption. Therefore \cite[Lemma 3]{FG}
can be extended to regular potentials of arbitrary polynomial growth
and \eqref{eiepsilon} also holds for $q>1$.
 }
\end{oss}

\section{Global attractor in 2D}
\setcounter{equation}{0} In this section
we first prove that in 2D we can define a generalized
semiflow on a suitable metric space $\mathcal{X}_{m_0}$ which is point
dissipative and eventually bounded. Furthermore, we show that such
generalized semiflow possesses a (unique) global attractor, provided
that the potential $F$ is bounded in $(-1,1)$ (like, e.g., \eqref{log}).
The argument is a generalization
of the one used in \cite{FG} and based on \cite{Ba}.
Henceforth, we refer to \cite{Ba} for the basic definitions and results
on the theory of generalized semiflows.

Consider system \eqref{eq1}-\eqref{eq4} endowed with
\eqref{eq5} for $d=2$ and assume that the external force $h$ is
time-independent, i.e.,
\begin{description}
\item[(A9)]  $h\in V_{div}'.$
\end{description}

The first step is to define a suitable metric space for the weak
    solutions and consequently to construct a
    generalized semiflow. To this aim, fix $m_0\in(0,1)$ and introduce
    the metric space
\begin{align}
&\mathcal{X}_{m_0}:=G_{div}\times\mathcal{Y}_{m_0},
\end{align}
where
\begin{align}
\label{phasesp2}
&\mathcal{Y}_{m_0}:=\{\varphi\in L^{\infty}(\Omega):
|\varphi|<1\mbox{ a.e. in }\Omega,\:\:\:F(\varphi)\in L^1(\Omega),
\:\:\:|\overline{\varphi}|\leq m_0\}.
\end{align}
The space $\mathcal{X}_{m_0}$ is endowed with the metric
\begin{align}
&\boldsymbol{d}(z_1,z_2):=
\Vert u_1-u_2\Vert+\Vert\varphi_1-\varphi_2\Vert+
\Big|\int_{\Omega}F(\varphi_1)-\int_{\Omega}F(\varphi_2)\Big|^{1/2},
\end{align}
for every $z_1:=[u_1,\varphi_1]$ and $z_2:=[u_2,\varphi_2]$ in $\mathcal{X}_{m_0}$.
Let us denote by $\mathcal{G}$ the set of all weak solutions corresponding to
all initial data $z_0=[u_0,\varphi_0]\in \mathcal{X}_{m_0}$.
We prove that $\mathcal{G}$ is a generalized semiflow on $\mathcal{X}_{m_0}$.
\begin{prop}
\label{gensemifl} Let $d=2$ and suppose that (A1)-(A7) and (A9)
hold. Then $\mathcal{G}$ is a generalized semiflow on
$\mathcal{X}_{m_0}$.
\end{prop}
\begin{proof}
It can be seen immediately that hypotheses (H1), (H2) and (H3) of
the definition of generalized semiflow \cite[Definition 2.1]{Ba} are
satisfied. It remains to prove the upper semicontinuity with respect to
initial data, i.e., that $\mathcal{G}$ satisfies (H4) of \cite[Definition
2.1]{Ba}. We can argue as in \cite[Proposition 3]{FG}. Thus we only
give the main steps of the proof. Consider a sequence
$\{z_j\}\subset\mathcal{G}$, with $z_j:=[u_j,\varphi_j]$ such that
$z_j(0):=[u_{j0},\varphi_{j0}]\to z_0:=[u_0,\varphi_0]$ in
$\mathcal{X}_{m_0}$. We have to show that there exist a subsequence
$\{z_{j_k}\}$ and a weak solution $z\in\mathcal{G}$ with
$z(0)=z_0$ such that $z_{j_k}(t)\to z(t)$ for each $t\geq 0$. Now,
every weak solution $z_j$ satisfies the energy identity
\eqref{idendiff} so that
\begin{align}
&
\mathcal{E}(z_j(t))+\int_0^t\Big(2\Vert\sqrt{\nu(\varphi_j)}Du_j(\tau)\Vert^2+\|\nabla\mu_j(\tau)\|^2\Big)d\tau
=\mathcal{E}(z_{j0})
+\int_0^t\langle h,u_j(\tau)\rangle d\tau,
\label{enidentity}
\end{align}
where $z_{j0}:=z_j(0)$. From this identity and using the assumptions
on  $F$ we deduce estimates of the form
\eqref{best1}-\eqref{dest3}. Furthermore, since
$|\varphi_{0j}|\leq m_0$ and $m_0\in(0,1)$ is fixed, we can repeat the
argument used in the existence proof to control the sequence of the
averages of the approximated chemical potentials (see
\eqref{Pe1abst}-\eqref{Ken2}) and get
\begin{align}
&\Vert F'(\varphi_j)\Vert_{L^1(\Omega)}\leq L_{m_0},
\label{wx}
\end{align}
where $L_{m_0}\in L^2(0,T)$. Hence, an estimate of the form
\eqref{best5} for $\mu_j$ holds. From these estimates we deduce
the existence of a couple $z=[u,\varphi]$ and of a function $\mu$
with $u$, $\varphi$ and $\mu$  having the regularity properties
\eqref{re1}-\eqref{prmu} and such that \eqref{c1}-\eqref{c8} hold
for suitable subsequences of $\{u_j\}$, $\{\varphi_j\}$ and
$\{\mu_j\}$. In order to prove that $z=[u,\varphi]$ is a weak
solution by passing to the limit in the variational formulation for $z_j$
we need to know that \eqref{re7} is satisfied for $\varphi$. To this
aim we use the same argument we applied to the sequence of
approximate solutions $\{\phie\}$ (cf. proof of Theorem
\ref{existence}).

More precisely, for $\eta\in(0,1)$ fixed and for a.a. fixed $t>0$, we can introduce the sets
$$E^j_{1,\eta}:=\{\varphi_j(x,t)>1-\eta\},\qquad E^j_{2,\eta}:=\{\varphi_j(x,t)>-1+\eta\},$$
and so we have
$$H'(1-\eta)|E^j_{1,\eta}|\leq\Vert H'(\varphi_j)\Vert_{L^1(\Omega)},
\qquad |H'(-1+\eta)||E^j_{2,\eta}|\leq\Vert H'(\varphi_j)\Vert_{L^1(\Omega)},$$
where $H$ is defined as in \eqref{H}. Therefore, recalling \eqref{wx}, by first letting
$j\to\infty$ and then $\eta\to 0$ we can deduce that
$$|\varphi(x,t)|<1\qquad\mbox{for a.e. }x\in\Omega\;\mbox{and for a.e. }t>0.$$
On the other hand, since we also have
$$u_j(t)\rightharpoonup u(t)\quad\mbox{weakly in }G_{div},\qquad
\varphi_j(t)\rightharpoonup\varphi(t)\quad\mbox{weakly in
}H,\qquad\forall t\geq 0,$$ then $z(0)=z_0$. It remains to prove the
convergence of the sequence $\{z_j(t)\}$ to $z(t)$ in
$\mathcal{X}_{m_0}$ for each $t\geq 0$. Reasoning as in \cite{FG},
we represent the singular potential F as follows
$$F(s)=G(x,s)-\Big(a(x)-\frac{c_0}{2}\Big)\frac{s^2}{2},$$
where $c_0=\alpha+\beta+\min_{[-1,1]}F_2''>0$. Here, due to
(A6), the function $G(x,\cdot)$ is strictly convex in $(-1,1)$ for a.e.
$x\in\Omega$. Therefore, the energy $\mathcal{E}$ can still be
written as
$$\mathcal{E}(z)=\frac{1}{2}\|u\|^2+\frac{c_0}{4}\|\varphi\|^2-\frac{1}{2}(\varphi,J\ast\varphi)
+\int_{\Omega}G(x,\varphi(x))dx,\qquad\forall
z=[u,\varphi]\in\mathcal{X}_{m_0},$$ and the same argument used in
\cite[Proposition 3]{FG} applies.
\end{proof}
As done for regular potentials (see \cite{FG}), a dissipativity property
of the generalized semiflow $\mathcal{G}$ can be proven in the case
of singular (bounded) potentials.
\begin{prop}
\label{dissipative} Let $d=2$ and suppose that (A1)-(A7), (A9) hold.
Then $\mathcal{G}$ is point dissipative and eventually bounded.
\end{prop}
\begin{proof}
Recalling the proof of \cite[Corollary 2]{CFG} a dissipative
estimate can be established, namely,
\begin{align}
\label{diss2D}
&
\mathcal{E}(z(t))\leq \mathcal{E}(z_0)e^{-kt}+ F(\overline{\varphi_0})|\Omega| + K,
\qquad\forall t\geq 0,
\end{align}
where $k$, $K$ are two positive constants
which are independent of the initial data, with $K$ depending on
$\Omega$, $\nu_1$, $J$, $F$, $\|h\|_{V_{div}'}$. From
\eqref{diss2D} we get (see \cite[Proposition 4]{FG})
$$\boldsymbol{d}^2(z(t),0)\leq c\mathcal{E}(z_0)e^{-kt}+cM_{m_0}+c,\qquad\forall t\geq 0,$$
which entails that the generalized semiflow $\mathcal{G}$ is point
dissipative and eventually bounded.
\end{proof}

We can now state the main result of this section.
\begin{prop}
Let $d=2$ and suppose that (A1)-(A7), (A9) hold. Furthermore, assume that $F$ is bounded
in $(-1,1)$.
Then $\mathcal{G}$ possesses a global attractor.
\end{prop}
\begin{proof}
In light of Proposition \ref{dissipative} and by \cite[Proposition 3.2]{Ba}
and \cite[Theorem 3.3]{Ba}
we only need to show that $\mathcal{G}$ is compact.
Let $\{z_j\}\subset\mathcal{G}$ be a sequence with $\{z_j(0)\}$ bounded in $\mathcal{X}_{m_0}$. We claim that
there exists a subsequence $\{z_{j_k}\}$ such that $z_{j_k}(t)$ converges in $\mathcal{X}_{m_0}$
for every $t>0$.
Indeed, the energy identity \eqref{enidentity} entails the existence of a subsequence (not relabeled)
such that (see the proof of Proposition \ref{gensemifl}) for almost all $t>0$
$$u_j(t)\to u(t)\quad\mbox{strongly in }G_{div},\qquad
\varphi_j(t)\to\varphi(t)\quad\mbox{strongly in }H\mbox{ and a.e. in }\Omega,$$
where $z=[u,\varphi]$ is a weak solution.
Since $F$ is bounded in $(-1,1)$, by Lebesgue's theorem we therefore have
$$\int_{\Omega}F(\varphi_j(t))\to\int_{\Omega}F(\varphi(t)),\qquad\mbox{a.e. }t>0.$$
Hence $\mathcal{E}(z_j(t))\to\mathcal{E}(z(t))$ for almost all
$t>0$. Thus, arguing as in \cite[Theorem 3, Proposition 3]{FG}, we
deduce that $z_j(t)\to z(t)$ in $\mathcal{X}_{m_0}$ {\itshape for all}
$t>0$, which yields the compactness of $\mathcal{G}$.
\end{proof}

We can also prove the existence of
the global attractor for the convective nonlocal Cahn-Hilliard equation
with $u\in L^{\infty}(\Omega)^d \cap V_{div}$, $d=2,3$. Indeed,
thanks to Corollary \ref{NLCH1}, we can define a semigroup $S(t)$ on $\mathcal{Y}_{m_0}$ (cf.
\eqref{phasesp2}) endowed the metric
\begin{equation*}
\bar{\mathbf{d}}(\varphi_1,\varphi_2)=
\|\varphi_1-\varphi_2\|+\Big|\int_{\Omega}
F(\varphi_1)-\int_{\Omega}F(\varphi_2)\Big|^{1/2},
\quad \forall\,\varphi_1,\varphi_2 \in \mathcal{Y}_{m_0}.
\end{equation*}
Then we have
\begin{thm}
\label{NLCHattr} Let $u\in L^{\infty}(\Omega)^d\cap V_{div}$ be
given. Suppose that (A1), (A3)-(A7) are satisfied and assume that $F$
is bounded in $(-1,1)$. Then the dynamical system
$(\mathcal{Y}_{m_0},S(t))$ possesses a connected global attractor.
\end{thm}

The proof goes as in \cite[Proof of Theorem 4]{FG}.

% % % % % % % % % % % % % % % % % % % % % % % % % % % % % % % % % % % % % % % % % % % % % % % % % % % % % % % % % % % % % % % % % % % % % % % % % % % % % % % % % % % % % % % % % % % % % % % % % % % % % % % % % % % % % % % % %
% % % % % % % % % % % % % % % % % % % % % % % % % % % % % % % % % % % % % % % % % % % % % % % % % % % % % % % % % % % % % % % % % % % % % % % % % % % % % % % % % % % % % % % % % % % % % % % % % % % % % % % % % % % % % % % % % % % % % % % % % % % % % % % % % % % % % % % % % % % % % % % % % % % % % % % % % % % % % % % % % % % % % % % % % % % %
\section{Existence of a trajectory attractor}
\setcounter{equation}{0} In this section, by relying on the theory
developed in \cite{CV,CV2} (see also \cite{Se}), we prove that a
trajectory attractor can be constructed for the nonlocal
Cahn-Hilliard-Navier-Stokes system \eqref{eq1}-\eqref{eq4} subject
to \eqref{eq5} with $F$ satisfying (A3)-(A7). The construction of the
trajectory attractor for problem \eqref{eq1}-\eqref{eq5} in the case
of regular potentials with polynomial growth has been done in
\cite{FG}. We concentrate on the 3D case.

%Since
We shall need a slightly more general functional setting than the one
devised in \cite{CV}. Indeed, in order to construct a trajectory
attractor without any boundedness assumption on the potential $F$,
we must define a family of bounded sets of
trajectories with a suitable attraction property.
Henceforth, we refer to \cite{CV} for the main definitions and
notation. The idea is to take a subspace $\mathcal{F}_b^+$ of the
space $\mathcal{F}_{loc}^+$ (where $\mathcal{F}_{loc}^+$ as
well as its topology $\Theta_{loc}^+$ are defined as in \cite{CV}) on
which a metric $d_{\mathcal{F}_b^+}$ is given and assume that the
trajectory space $\mathcal{K}_{\sigma}^+$ corresponding to the
symbol $\sigma\in\Sigma$ satisfies
$\mathcal{K}_{\sigma}^+\subset\mathcal{F}_b^+$, for every
$\sigma\in\Sigma$. This approach is in the spirit of the theory of
$(\mathcal{M},\mathcal{T})-$attractors in \cite[Chap. XI, Section
3]{CV2}, where $\mathcal{T}$ is a topological space where some
metric is defined and $\mathcal{M}$ is the corresponding metric
space.

Consider the united trajectory space
$\mathcal{K}_{\Sigma}^+:=\cup_{\sigma\in\Sigma}\mathcal{K}_{\sigma}^+$
of the family $\{\mathcal{K}_{\sigma}^+\}_{\sigma\in\Sigma}$.
We have $\mathcal{K}_{\Sigma}^+\subset \mathcal{F}_b^+$  and
if the family $\{\mathcal{K}_{\sigma}^+\}_{\sigma\in\Sigma}$ is
translation-coordinated then we have
$T(t)\mathcal{K}_{\Sigma}^+\subset\mathcal{K}_{\Sigma}^+$,
for every $t\geq 0$, i.e., the translation semigroup $\{T(t)\}$ acts on
$\mathcal{K}_{\Sigma}^+$. Introduce now the family
$$\mathcal{B}_{\Sigma}:=\left\{B\subset\mathcal{K}_{\Sigma}^+:B\mbox{ bounded in } \mathcal{F}_b^+
\mbox{ w.r.t. the metric }d_{\mathcal{F}_b^+}\right\}.$$
We shall refer to this family in the definition
of a uniformly (w.r.t $\sigma\in\Sigma$) attracting set $P\subset\mathcal{F}_{loc}^+$
for $\{\mathcal{K}_{\sigma}^+\}_{\sigma\in\Sigma}$ in the topology $\Theta_{loc}^+$
and in the definition of the uniform (w.r.t. $\sigma\in\Sigma$) trajectory attractor
$\mathcal{A}_{\Sigma}$ of the translation semigroup $\{T(t)\}$.

To prove some properties of the trajectory attractor we need that the
set $\mathcal{K}_{\Sigma}^+$ be closed in $\Theta_{loc}^+$.
Recall that the family
$\{\mathcal{K}_{\sigma}^+\}_{\sigma\in\Sigma}$ is called
$(\Theta_{loc}^+,\Sigma)-$closed if the graph set
$\cup_{\sigma\in\Sigma}\mathcal{K}_{\sigma}^+\times\{\sigma\}$
is closed in the topological space $\Theta_{loc}^+\times\Sigma$. If
$\{\mathcal{K}_{\sigma}^+\}_{\sigma\in\Sigma}$ is
$(\Theta_{loc}^+,\Sigma)-$closed and $\Sigma$ is compact, then
$\mathcal{K}_{\Sigma}^+$ is closed in $\Theta_{loc}^+$.

\begin{oss}
\upshape{ We shall see that (cf. Proposition \ref{trajclosedness}),
although by means of the topological-metric scheme above the
boundedness assumption on the potential $F$ can be avoided as far as
the construction of the trajectory attractor for system
\eqref{eq1}-\eqref{eq5} with singular potential is concerned, it
seems difficult to get rid of such an assumption when one wants to prove
the closedness of the trajectory space $\mathcal{K}_{\Sigma}^+$. }
\end{oss}

We now state the main abstract result which
can be established by applying \cite[Chap. XI, Theorem 2.1]{CV2} to the topological
space $\mathcal{F}_{loc}^+$, to the family $\mathcal{B}_{\Sigma}$
and to the family
$$\mathcal{B}_{\omega(\Sigma)}:=\left\{B\subset\mathcal{K}_{\omega(\Sigma)}^+:B\mbox{ bounded in } \mathcal{F}_b^+
\mbox{ w.r.t. the metric }d_{\mathcal{F}_b^+}\right\},$$
where $\mathcal{K}_{\omega(\Sigma)}^+:=\cup_{\sigma\in\omega(\Sigma)}\mathcal{K}_{\sigma}^+$
and where $\omega(\Sigma)$ is the $\omega-$limit set of $\Sigma$,
(see also \cite[Theorem 3.1]{CV}).
\begin{thm}
\label{trajattract}
Let the spaces $(\mathcal{F}_{loc}^+,\Theta_{loc}^+)$ and $(\mathcal{F}_b^+,d_{\mathcal{F}_b^+})$
be as above, and the family of trajectory spaces $\{\mathcal{K}_{\sigma}^+\}_{\sigma\in\Sigma}$
corresponding to the evolution equation with symbols $\sigma\in\Sigma$ be such that
$\mathcal{K}_{\sigma}^+\subset\mathcal{F}_b^+$, for every $\sigma\in\Sigma$.
%Let $(\mathcal{F}_{loc}^+,\Theta_{loc}^+)$ be a Hausdorff space with a countable base
%and let $\mathcal{F}_b^+$ be a subspace of $\mathcal{F}_{loc}^+$ which
%is metric with metric $d_{\mathcal{F}_b^+}$.
%Suppose that the family of trajectory spaces $\{\mathcal{K}_{\sigma}^+\}_{\sigma\in\Sigma}$
%corresponding to an evolution equation with symbols $\sigma\in\Sigma$ satisfies
%$\mathcal{K}_{\sigma}^+\subset\mathcal{F}_b^+$, for every $\sigma\in\Sigma$.
Assume there exists a subset $P\subset\mathcal{F}_{loc}^+$
which is compact in $\Theta_{loc}^+$ and uniformly (w.r.t. $\sigma\in\Sigma$)
attracting in $\Theta_{loc}^+$
for the family $\{\mathcal{K}_{\sigma}^+\}_{\sigma\in\Sigma}$ in the
topology $\Theta_{loc}^+$.
% of $d_{\mathcal{F}_{loc}^+}-$bounded subsets of $\mathcal{K}_{\Sigma}^+$.
Then, the translation semigroup $\{T(t)\}_{t\geq 0}$, which acts on
$\mathcal{K}_{\Sigma}^+$ if the family
$\{\mathcal{K}_{\sigma}^+\}_{\sigma\in\Sigma}$ is
translation-coordinated, possesses a (unique) uniform (w.r.t.
$\sigma\in\Sigma$) trajectory attractor
$\mathcal{A}_{\Sigma}\subset P$.
%for the family of $d_{\mathcal{F}_{loc}^+}-$bounded subsets of $\mathcal{K}_{\Sigma}^+$.
If the semigroup $\{T(t)\}_{t\geq 0}$ is continuous in $\Theta_{loc}^+$, then $\mathcal{A}_{\Sigma}$
is strictly invariant
$$T(t)\mathcal{A}_{\Sigma}=\mathcal{A}_{\Sigma},\qquad\forall t\geq 0.$$
In addition, if the family
$\{\mathcal{K}_{\sigma}^+\}_{\sigma\in\Sigma}$ is
translation-coordinated and $(\Theta_{loc}^+,\Sigma)-$closed, with
$\Sigma$ a compact metric space, then
$\mathcal{A}_{\Sigma}\subset\mathcal{K}_{\Sigma}^+$ and
$$\mathcal{A}_{\Sigma}=\mathcal{A}_{\omega(\Sigma)},$$
where $\mathcal{A}_{\omega(\Sigma)}$ is the uniform (w.r.t. $\sigma\in\omega(\Sigma)$) trajectory attractor
for the family $\mathcal{B}_{\omega(\Sigma)}$
%:of $d_{\mathcal{F}_{loc}^+}-$bounded subsets of $\mathcal{K}_{\omega(\Sigma)}^+$
and $\mathcal{A}_{\omega(\Sigma)}\subset\mathcal{K}_{\omega(\Sigma)}^+$.
\end{thm}

Suppose that for a given abstract nonlinear non-autonomous evolution
equation a dissipative estimate of the following form can be
established
\begin{align}
&d_{\mathcal{F}_b^+}(T(t)w,w_0)\leq\Lambda_0\Big(d_{\mathcal{F}_b^+}(w,w_0)\Big)e^{-kt}+\Lambda_1,
\qquad\forall t\geq t_0,
\label{abstdiss}
\end{align}
for every $w\in\mathcal{K}_{\Sigma}^+$, for some fixed
$w_0\in\mathcal{F}_b^+$ and for some
$\Lambda_0:[0,\infty)\to[0,\infty)$ locally bounded and some
constants $\Lambda_1\geq 0$, $k>0$, where $k$, $\Lambda_0$ and
$\Lambda_1$ are independent of $w$. Furthermore, assume that the
ball
$$B_{\mathcal{F}_b^+}(w_0,2\Lambda_1):=\{w\in\mathcal{F}_b^+:d_{\mathcal{F}_b^+}(w,w_0)\leq 2\Lambda_1\}$$
is compact in $\Theta_{loc}^+$. By virtue of \eqref{abstdiss} such
ball is a uniformly (w.r.t. $\sigma\in\Sigma$) attracting set for the
family $\{\mathcal{K}_{\sigma}^+\}_{\sigma\in\Sigma}$ in the
topology $\Theta_{loc}^+$ (actually,
$B_{\mathcal{F}_b^+}(w_0,2\Lambda_1)$ is uniformly (w.r.t.
$\sigma\in\Sigma$) absorbing for the family
$\mathcal{B}_{\Sigma}$). Theorem \ref{trajattract} therefore
entails that the translation semigroup $\{T(t)\}_{t\geq 0}$
possesses a (unique) uniform (w.r.t. $\sigma\in\Sigma$) trajectory
attractor $\mathcal{A}_{\Sigma}\subset
B_{\mathcal{F}_b^+}(w_0,2\Lambda_1)$. $\\$ Let us now turn to
\eqref{eq1}-\eqref{eq5} and apply to this system the scheme
described above. $\\$ For $q\geq 1$, $m_0\in(0,1)$ and for any
given $M>0$ we set
\begin{align*}
\mathcal{F}_M =\Big\{&[v,\psi] \in L^\infty(0,M;G_{div}\times L^{2+2q}(\Omega))\cap L^2(0,M;V_{div}\times V)\,:\\
&v_t \in L^{4/3}(0,M;V^\prime_{div}),\; \psi_t \in L^2(0,M;V^\prime),\nonumber\\
&\psi\in L^{\infty}(Q_M), \; |\psi|<1\;\mbox{a.e. in }Q_M, \; |\overline{\psi}|\leq m_0\Big\},
\end{align*}
where $Q_M=\Omega\times(0,M)$. We endow $\mathcal{F}_M$ with the weak topology $\Theta_M$ which induces the following notion of weak convergence:
a sequence $\{[v_m,\psi_m]\}\subset\mathcal{F}_M$ is said to converge
to $[v,\psi]\in\mathcal{F}_M$ in $\Theta_M$ if
\begin{align*}
&v_n\rightharpoonup v\quad
\mbox{weakly}^{\ast}\mbox{ in }L^{\infty}(0,M;G_{div})
\;\mbox{and weakly in }L^2(0,M;V_{div}),\\
&(v_n)_t\rightharpoonup v_t\quad\mbox{weakly in }L^{4/3}(0,M;V_{div}'),\\
&\psi_n\rightharpoonup\psi\quad
\mbox{weakly}^{\ast}\mbox{ in }L^{\infty}(0,M;L^{2+2q}(\Omega))
\;\mbox{and weakly in }L^2(0,M;V),\\
&(\psi_n)_t\rightharpoonup\psi_t\quad\mbox{weakly in }L^2(0,M;V').
\end{align*}
Then, we can define the space
\begin{align*}
\mathcal{F}^+_{loc} =\Big\{&[v,\psi] \in L^\infty_{loc}([0,\infty);G_{div}\times L^{2+2q}(\Omega))\cap L^2_{loc}([0,\infty); V_{div}\times V)\,:\\
&v_t \in L^{4/3}_{loc}([0,\infty);V^\prime_{div}),\; \psi_t \in L^2_{loc}([0,\infty);V^\prime),\\
&\psi\in L^{\infty}(Q_M),\;
|\psi|<1\;\mbox{a.e. in }Q_M,\;\forall M>0,\;|\overline{\psi}|\leq m_0\Big\},
\end{align*}
endowed with the inductive limit weak topology $\Theta_{loc}^+$. In
$\mathcal{F}^+_{loc}$ we consider the following subset
\begin{align*}
 \mathcal{F}^+_{b} =\Big\{&[v,\psi] \in L^\infty(0,\infty;G_{div}\times L^{2+2q}(\Omega))
 \cap L^2_{tb}(0,\infty; V_{div}\times V)\,:\\
 &v_t \in L^{4/3}_{tb}(0,\infty;V^\prime_{div}),\; \psi_t \in L^2_{tb}(0,\infty;V^\prime),\\
&\psi\in L^{\infty}(Q_{\infty}),\;
|\psi|<1\;\mbox{a.e. in }Q_{\infty},\;|\overline{\psi}|\leq m_0,\;
F(\psi)\in L^{\infty}(0,\infty;L^1(\Omega))\Big\},
 \end{align*}
where $Q_{\infty}:=\Omega\times(0,\infty)$, endowed with the
following metric
\begin{align}
d_{\mathcal{F}^+_{b}}(z_2,z_1):&=\|z_2-z_1\|_{L^{\infty}(0,\infty;G_{div}\times L^{2+2q}(\Omega))}+
\|z_2-z_1\|_{L^2_{tb}(0,\infty;V_{div}\times V)}\nonumber\\
&+\|(v_2)_t-(v_1)_t\|_{L^{4/3}_{tb}(0,\infty;V^\prime_{div})}
+\|(\psi_2)_t-(\psi_1)_t\|_{L^2_{tb}(0,\infty;V^\prime)}\nonumber\\
&+\Big\|\int_{\Omega}F(\psi_2)-\int_{\Omega}F(\psi_1)\Big\|_{L^{\infty}(0,\infty)}^{1/2},
\end{align}
for all $z_2:=[v_2,\psi_2]$,
$z_1:=[v_1,\psi_1]\in\mathcal{F}^+_{b}$. Here we recall that
$L^p_{tb}(0,\infty;X)$, $p\geq 1$ and $X$ being a Banach space, is
the Banach space of the translation bounded functions (see, e.g.,
\cite{CV2}).

For the trajectory space $\mathcal{K}_h^+$ corresponding to a
symbol $h$ we mean
\begin{defn}
For every $h\in L^2_{loc}([0,\infty);V_{div}')$ the trajectory space
$\mathcal{K}_{h}^+$ is the set of all weak solutions $z=[v,\psi]$ to
\eqref{eq1}-\eqref{eq5} with external force $h$ which belong to the
space $\mathcal{F}_{loc}^+$ and satisfy the energy inequality
\eqref{ei} for all $t\geq s$ and for a.a. $s\in(0,\infty)$.
\end{defn}
%\begin{oss}
%{\upshape
%We point out that in the definition of the trajectory space $\mathcal{K}_{h}^+$
%a bound on the average $\overline{\psi}$ of
%the weak solutions $z=[v,\psi]$ is assumed, namely
%$$|\overline{\psi}|\leq m_0,\qquad\forall z=[v,\psi]\in\mathcal{K}_{h}^+.$$
%}
%\end{oss}
\begin{oss}
{\upshape
Notice that in the definition of the trajectory space
$\mathcal{K}_{h}^+$ we do not assume that the energy
inequality \eqref{ei} is satisfied also for $s=0$.
In this way the family $\{\mathcal{K}_{h}^+\}_{h\in\Sigma}$
($\Sigma$ is a generic symbol space included in $L^2_{loc}([0,\infty);V_{div}')$)
is translation-coordinated and therefore the semigroup $\{T(t)\}$ acts on
$\mathcal{K}_{\Sigma}^+$.
}
\end{oss}

According to Theorem \ref{existence}, if (A1)-(A7) hold, then for
every $z_0=[v_0,\psi_0]$ such that
$$v_0\in G_{div},\qquad\psi_0\in L^{\infty}(\Omega),\qquad\|\psi_0\|_{\infty}\leq 1,
\qquad F(\psi_0)\in L^1(\Omega),$$
and every $h$ satisfying (A8) there exists a trajectory $z\in\mathcal{K}_{h}^+$
for which $z(0)=z_0$.

Let us consider now
$$h_0\in L^2_{tb}(0,\infty;V_{div}'),$$
and observe that $h_0$ is translation compact in
$L^2_{loc,w}([0,\infty);V_{div}')$ (see, e.g., \cite[Proposition
6.8]{CV}). As symbol space $\Sigma$ we take the compact metric
space given by the hull of $h_0$ in
$L^2_{loc,w}([0,\infty);V_{div}')$
$$\Sigma=\mathcal{H}_+(h_0):=[\{T(t)h_0:t\geq 0\}]_{L^2_{loc,w}([0,\infty);V_{div}')},$$
where $[\cdot]_X$ denotes the closure in $X$. Recall that every
$h\in\mathcal{H}_+(h_0)$ is translation compact in
$L^2_{loc,w}([0,\infty);V_{div}')$ as well (see \cite[Proposition
6.9]{CV}) and
\begin{align}
&\Vert h\Vert_{L^2_{tb}(0,\infty;V_{div}')}\leq\Vert h_0\Vert_{L^2_{tb}(0,\infty;V_{div}')},
\qquad\forall h\in\mathcal{H}_+(h_0).
\label{tb}
\end{align}

Hence we can state the main result of this section.
\begin{thm}
\label{trattex} Let (A1)-(A7) hold and assume $h_0\in
L^2_{tb}(0,\infty;V_{div}')$. Then, the translation semigroup
$\{T(t)\}$ acting on $\mathcal{K}_{\mathcal{H}_+(h_0)}^+$
possesses the uniform (w.r.t.
$h\in\mathcal{K}_{\mathcal{H}_+(h_0)}^+$) trajectory attractor
$\mathcal{A}_{\mathcal{H}_+(h_0)}$. This set is strictly invariant,
bounded in $\mathcal{F}_b^+$ and compact in $\Theta_{loc}^+$.
In addition, if the potential $F$ is bounded on $(-1,1)$, then
$\mathcal{K}_{\mathcal{H}_+(h_0)}^+$ is closed in $\Theta_{loc}^+$,
$\mathcal{A}_{\mathcal{H}_+(h_0)}\subset\mathcal{K}_{\mathcal{H}_+(h_0)}^+$
and we have
$$\mathcal{A}_{\mathcal{H}_+(h_0)}=\mathcal{A}_{\omega(\mathcal{H}_+(h_0))}.$$
\end{thm}

The proof of Theorem \ref{trattex} is based on two propositions. The
first one establishes a dissipative estimate of the form
\eqref{abstdiss} for weak solutions to \eqref{eq1}-\eqref{eq5}.

\begin{prop}
\label{dissprop} Let (A1)-(A7) hold and let $h_0\in
L^2_{tb}(0,\infty;V_{div}')$. Then, for all
$h\in\mathcal{H}_+(h_0)$, we have
$\mathcal{K}_h^+\subset\mathcal{F}_b^+$ and the following
dissipative estimate holds
\begin{align}
d_{\mathcal{F}_b^+}(T(t)z,0)\leq\Lambda_0\Big(d_{\mathcal{F}_b^+}(z,0)\Big)e^{-kt}+\Lambda_1,
\qquad\forall t\geq 1,
\label{tradissest}
\end{align}
for all $z\in\mathcal{K}_h^+$. Here $\Lambda_0:[0,\infty)\to[0,\infty)$ is a nonnegative
monotone increasing continuous function, $k$ and $\Lambda_1$ are two positive
constants with $k=\min(1/2,\lambda_1\nu_1)$,
$\lambda_1$ being the first
eigenvalue of the Stokes operator $S$. Moreover,
$\Lambda_0$, $\Lambda_1$ depend
on $\nu_1,\nu_2,\lambda_1,F,J,|\Omega|$,
and $\Lambda_1$ also depends on
$\Vert h_0\Vert_{L^2_{tb}(0,\infty;V_{div}')}$ and on $m_0$.
\end{prop}

\begin{proof}
The following estimate can be obtained by arguing as in the proof of
\cite[Corollary 2]{CFG} (see also the proof of \cite[Theorem
5]{FG}). There exist two positive constants $k_1$, $k_2$ such that
\begin{align}
&\mathcal{E}(z)\leq k_1\Big(\frac{\nu_1}{2}\Vert\nabla v\Vert^2+\Vert\nabla\mu\Vert^2\Big)+k_2,
\label{auxest}
\end{align}
for every weak solution $z=[v,\psi]$ to  \eqref{eq1}-\eqref{eq5}
satisfying $\overline{\psi}=0$. Furthermore, it can be shown that
$k_1=\max(2,1/\lambda_1\nu_1)$.

Take now $z=[v,\psi]\in\mathcal{K}_h^+$ with
$h\in\mathcal{H}_+(h_0)$ and set
$\widetilde{z}=[v,\widetilde{\psi}]$, where
$\widetilde{\psi}:=\psi-\overline{\psi}$. Recall that
$\overline{\psi}=\overline{\psi_0}$. It is easily seen that
$\widetilde{z}$ is a weak solution to the same system where the
potential $F$ and the viscosity $\nu$ are replaced by, respectively,
$$\widetilde{F}(s):=F(s+\overline{\psi}_0)-F(\overline{\psi_0}),\qquad\widetilde{\nu}(s):=\nu(s+\overline{\psi_0}).$$
Since $z$ satisfies \eqref{ei} for all $t\geq s$ and for a.a.
$s\in(0,\infty)$, then an energy inequality of the same form as
\eqref{ei} also holds for $\widetilde{z}$, namely,
\begin{align}
& &\widetilde{\mathcal{E}}(\widetilde{z}(t))+\int_s^t(2\|\sqrt{\widetilde{\nu}(\widetilde{\psi})}Dv\|^2
+\|\nabla\widetilde{\mu}\|^2)d\tau\leq\widetilde{\mathcal{E}}(\widetilde{z}(s))+\int_s^t\langle h(\tau),
v(\tau)\rangle d\tau,
\label{eisttilde}
\end{align}
for all $t\geq s$ and for a.a. $s\in(0,\infty)$, where we have set
$$\widetilde{\mathcal{E}}(\widetilde{z}(t)):
=\frac{1}{2}\|v(t)\|^2
+\frac{1}{4}\int_{\Omega}\int_{\Omega}J(x-y)(\widetilde{\psi}(x,t)-\widetilde{\psi}(y,t))^2
dxdy+\int_{\Omega}\widetilde{F}(\widetilde{\psi}(t))$$ and
$\widetilde{\mu}:=a\widetilde{\psi}-J\ast\widetilde{\psi}+\widetilde{F}'(\widetilde{\psi})
=a\psi-J\ast\psi+F'(\psi)=\mu$. The weak solution $\widetilde{z}$
fulfills $(\widetilde{\psi},1)=0$ and therefore \eqref{auxest} can be
applied to $\widetilde{z}$. Such estimate and \eqref{eisttilde} entail
the inequality
\begin{align*}
&\widetilde{\mathcal{E}}(\widetilde{z}(t))
 +\frac{1}{k_1}\int_0^t\widetilde{\mathcal{E}}(\widetilde{z}(\tau))d\tau
\leq \frac{k_2}{k_1}(t-s)+\frac{1}{2\nu_1}\int_s^t\|h(\tau)\|_{V_{div}'}^2d\tau\nonumber\\
&+\widetilde{\mathcal{E}}(\widetilde{z}(s))
 +\frac{1}{k_1}\int_0^s\widetilde{\mathcal{E}}(\widetilde{z}(\tau))d\tau,\qquad
\forall t\geq s,\quad\mbox{a.a. }s\in(0,\infty).
 \end{align*}
By means of the identity
 \begin{align*}
 &\widetilde{\mathcal{E}}(\widetilde{z}(t))=\mathcal{E}(z(t))-F(\overline{\psi_0})|\Omega|,
 \end{align*}
from the previous inequality we get
\begin{align}
&\mathcal{E}(z(t))+k\int_0^t\mathcal{E}(z(\tau))
 d\tau\leq l(t-s)+\frac{1}{2\nu_1}\int_s^t\|h(\tau)\|_{V_{div}'}^2d\tau
 +\mathcal{E}(z(s))
 +k\int_0^s\mathcal{E}(z(\tau))d\tau,
 \label{intGronw}
\end{align}
for all $t\geq s$ and for a.a. $s\in(0,\infty)$, where $k=1/k_1$ and
$l=k_2/k_1+F(\overline{\psi_0})|\Omega|/k_1$. By applying \cite[Lemma 1]{FG}
from \eqref{intGronw} we deduce that
\begin{align}
&\mathcal{E}(z(t))
\leq \mathcal{E}(z(s))e^{-k(t-s)}
+\frac{1}{2\nu_1}\int_s^t e^{-k(t-\tau)}\left(\|h(\tau)\|_{V_{div}'}^2 + 2\nu_1 l \right)d\tau\nonumber\\
&\leq e^k\sup_{s\in(0,\infty)}\mathcal{E}(z(s))e^{-kt}+K^2,
\label{ta1}
\end{align}
for all $t\geq0$ and for a.a. $s\in(0,\infty)$, where
$$K^2=\frac{l}{k}+\frac{l}{2\nu_1(1-e^{-k})}
\Vert h_0\Vert_{L^2_{tb}(0,\infty;V_{div}')}^2.$$ Here we have
used \eqref{tb}. Note that $|\overline{\psi_0}|\leq m_0$ and therefore $K$ can be
estimated by a constant depending on
$\nu_1,\lambda_1,F,J,|\Omega|$ and on $h_0$, $m_0$. Observe
now that we have
\begin{align}
&C_1\Big(\Vert v(s)\Vert^2+\Vert\psi(s)\Vert_{L^{2+2q}(\Omega)}^{2+2q}
+\int_{\Omega}F(\psi(s))-1\Big)\nonumber\\
&\leq\mathcal{E}(z(s))\leq C_2\Big(\Vert v(s)\Vert^2+\Vert\psi(s)\Vert_{L^{2+2q}(\Omega)}^{2+2q}
+\int_{\Omega}F(\psi(s))+1\Big),
\label{ta2}
\end{align}
and therefore
\begin{align}
\sup_{s\in(0,\infty)}\mathcal{E}(z(s))&\leq
C_2\Big(\Vert v\Vert_{L^{\infty}(0,1;G_{div})}^2+\Vert\psi\Vert_{L^{\infty}(0,1;L^{2+2q}(\Omega))}^{2+2q}
+\sup_{s\in(0,1)}\int_{\Omega}F(\psi(s))+1\Big)\nonumber\\
&\leq C_3 d_{\mathcal{F}_b^+}^{2+2q}(z,0).
\label{ta3}
\end{align}
By combining \eqref{ta1} with \eqref{ta2} and \eqref{ta3} we get
\begin{align}
&\Vert v(t)\Vert^2+\Vert\psi(t)\Vert_{L^{2+2q}(\Omega)}^{2+2q}
+\int_{\Omega}F(\psi(t))\leq
c d_{\mathcal{F}_b^+}^{2+2q}(z,0)e^{-kt}+K^2+c,\qquad\forall t\geq 1,
\label{jw}
\end{align}
which yields
\begin{align}
&\Vert T(t)v\Vert_{L^{\infty}(0,\infty;G_{div})}^2
+\Vert T(t)\psi\Vert_{L^{\infty}(0,\infty;L^{2+2q}(\Omega))}^{2+2q}
+\Big\Vert\int_{\Omega}F(T(t)\psi)\Big\Vert_{L^{\infty}(0,\infty)}\nonumber\\
&\leq c d_{\mathcal{F}_b^+}^{2+2q}(z,0)e^{-kt}+K^2+c,\qquad\forall t\geq 1.
\label{3terms}
\end{align}
On account of the definition of the metric $d_{\mathcal{F}_b^+}$,
\eqref{3terms} allows to estimate three terms on the left hand side of
\eqref{tradissest}. The remaining four terms on the left hand side of
\eqref{tradissest} can be handled by performing the same kind of
calculations done in the proof of \cite[Proposition 7]{FG}. In
particular, the two terms in the $L^2_{tb}(0,\infty;V_{div})$-norm of
$T(t)v$ and in the $L^2_{tb}(0,\infty;V)$-norm of $T(t)\psi$ can be
estimated by writing the energy inequality between $t$ and $t+1$ and
by using the estimate
$$\Vert\nabla\mu\Vert^2\geq k_3\Vert\nabla\psi\Vert^2-k_4\Vert\psi\Vert^2,$$
where $k_3=c_0^4/2$ and $k_4=2\Vert\nabla J\Vert_{L^1}^2$,
with $c_0=\alpha+\beta+\min_{[-1,1]}F_2''>0$.  This last estimate
has been obtained in \cite{CFG} for the case of regular potentials, but
it still holds for singular potentials satisfying assumption (A6). Finally,
the two terms in the $L^{4/3}_{tb}(0,\infty;V_{div}')$-norm of
$T(t)v_t$ and in the $L^2_{tb}(0,\infty;V')$-norm of $T(t)\psi_t$
can be estimated by comparison on account of
\eqref{jw}, using also the estimates for the
$L^2_{tb}(0,\infty;V_{div})$-norm of $T(t)v$ and the
$L^2_{tb}(0,\infty;V)$-norm of $T(t)\psi$. We refer to
\cite[Proposition 7]{FG} for the details.
\end{proof}

The next proposition, which concerns with the $(\Theta_M,L^2(0,M;V_{div}'))$-closedness
property of the family $\{\mathcal{K}_h^M\}_{h\in L^2(0,M;V_{div}')}$
of trajectory spaces on $[0,M]$, requires a boundedness assumption on the potential $F$.

\begin{prop}
\label{trajclosedness} Let (A1)-(A7) hold and assume that the
potential $F$ is bounded on $(-1,1)$. Let $h_m\in
L^2(0,M;V_{div}')$ and consider
$[v_m,\psi_m]\in\mathcal{K}_{h_m}^M$ such that
$\{[v_m,\psi_m]\}$ converges to $[v,\psi]$ in $\Theta_M$ and
$\{h_m\}$ converges to $h$ strongly in $L^2(0,M;V_{div}')$. Then
$[v,\psi]\in\mathcal{K}_h^M$.
\end{prop}
\begin{proof}
Observe that $[v_m,\psi_m]\in\mathcal{K}_{h_m}^+$

\item{(i)} belongs to $\mathcal{F}_M$ with $\mu_m$ satisfying \eqref{prmu};

\item{(ii)} fulfills
\eqref{weakfor1}-\eqref{weakfor2} together with $\mu_m=a\psi_m-J\ast\psi_m+F'(\psi_m)$
and $h=h_m$;

\item{(iii)} satisfies the energy inequality
\begin{align}
&\mathcal{E}(z_m(t))+\int_s^t(2\|\sqrt{\nu(\psi_m)}Dv_m\|^2+\|\nabla\mu_m\|^2)d\tau\leq\mathcal{E}(z_m(s))
+\int_s^t\langle h_m(\tau),
v_m(\tau)\rangle d\tau,
\label{eim}
\end{align}
for each $m\in\mathbb{N}$, for a.a. $s\in[0,M]$ and for all
$t\in[0,M]$ with $t\geq s$. Thus, due to the convergence assumption on
the sequence  $\{[v_m,\psi_m]\}$ and to the boundedness of $F$, it
is immediate to see that there exists a constant $c>0$ such that
\begin{equation}
\label{encontrol}
|\mathcal{E}(z_m(s))|\leq c,\qquad\forall m, \; \mbox{a.a. }s\in[0,M].
\end{equation}
Therefore, \eqref{eim} and the convergence assumption on the
sequence $\{h_m\}$ imply the control
$\Vert\nabla\mu_m\Vert_{L^2(0,M;H)}\leq c$. On the other hand,
by exploiting the argument used in the proof of Theorem
\ref{existence} it is easy to find the bound
$$\Vert F'(\psi_m)\Vert_{L^1(\Omega)}\leq L_{\overline{\psi_m}},$$
with $L_{\overline{\psi_m}}\in L^2(0,M)$ and furthermore we also
have $|\overline{\psi_m}|\leq m_0$, with $m_0\in(0,1)$.
Therefore, noting that
$\int_{\Omega}\mu_m=\int_{\Omega}F'(\psi_m)$, we deduce that
$\Vert\overline{\mu_m}\Vert_{L^2(0,M)}\leq c$, with the constant
$c$ depending on the fixed parameter $m_0$. The
Poincar\'e-Wirtinger inequality then implies
\begin{align}
&\Vert\mu_m\Vert_{L^2(0,M;V)}\leq c.
\end{align}
As a consequence, there exists $\mu\in L^2(0,M;V)$ such that up to a subsequence we have
\begin{align}
&\mu_m\rightharpoonup\mu,\qquad\mbox{weakly in }L^2(0,M;V).
\label{hhx}
\end{align}
Since, as a consequence of the convergence assumption on
$\{[v_m,\psi_m]\}$, for a subsequence we have
$[v_m,\psi_m]\to[v,\psi]$ strongly in $L^2(0,M;G_{div}\times H)$
and hence $\psi_m\to\psi$ also almost everywhere in
$\Omega\times(0,M)$, then we get $\mu=a\psi-J\ast\psi+F'(\psi)$.
Using now the convergence assumptions on $\{[v_m,\psi_m]\}$ and
on $\{h_m\}$, the above mentioned strong convergence and
\eqref{hhx}, we can pass to the limit in the variational formulation for
the weak solution $[v_m,\psi_m]$ with external force $h_m$ and
deduce that $[v,\psi]$ is a weak solution with external force $h$.

Finally, in order to prove that the weak solution $[v,\psi]$ satisfies
the energy inequality on $[0,M]$ with external force $h$ we let
$m\to\infty$ in \eqref{eim}.  In particular, we rely on the
convergence
$\sqrt{\nu(\psi_m)}Dv_m\rightharpoonup\sqrt{\nu(\psi)}Dv$
weakly in $L^2(0,M;H)$ (cf. \eqref{ddhh}) and on Lebesgue's
theorem to pass to the limit in the nonlinear term
$\int_{\Omega}F(\psi_m(s))$. Hence we conclude that
$[v,\psi]\in\mathcal{K}^M_h$.
\end{proof}

\begin{oss}
\upshape{It is not difficult to see, by arguing as in \cite[Chap. XV,
Prop. 1.1]{CV2}, that the same conclusion of Proposition
\ref{trajclosedness} holds if the convergence assumption on
$\{h_m\}$ is replaced with the weak convergence
$h_m\rightharpoonup h$  in $L^2(0,M;G_{div})$. }
\end{oss}

\begin{proof}[Proof of Theorem \ref{trattex}]
In virtue of Proposition \ref{dissprop} the ball
$B_{\mathcal{F}_b^+}(0,2\Lambda_0):=\{z\in\mathcal{F}_b^+:d_{\mathcal{F}_b^+}(z,0)\leq
2\Lambda_0\}$ is a uniformly (w.r.t. $h\in\mathcal{H}_+(h_0)$)
absorbing set for the family
$\{\mathcal{K}_h^+\}_{h\in\mathcal{H}_+(h_0)}$. Such a ball is
also precompact in $\Theta_{loc}^+$. By applying the first part of
Theorem \ref{trajattract} we deduce the existence of the uniform
(w.r.t. $h\in\mathcal{H}_+(h_0)$) trajectory attractor
$\mathcal{A}_{\mathcal{H}_+(h_0)} \subset
B_{\mathcal{F}_b^+}(0,2\Lambda_0)$, which is compact in
$\Theta_{loc}^+$ and, since $T(t)$ is continuous in
$\Theta_{loc}^+$, strictly invariant. Proposition \ref{trajclosedness}
and the fact that $\mathcal{H}_+(h_0)$ is a compact metric space
imply that the united trajectory space
$\mathcal{K}_{\mathcal{H}_+(h_0)}^+$ is closed in
$\Theta_{loc}^+$. The second part of Theorem \ref{trajattract}
allows us to conclude the proof.
\end{proof}

% % % % % % % % % % % % % % % % % % % % % % % % % % % % % % % % % % % % % % % % % % % % % % % % % % % % % % % % % % % % % % % % % % % % % % % % % % % % % % % % % % % % % % % % % % % % % % % % % % % % % % % % % % % % % % % % %
% % % % % % % % % % % % % % % % % % % % % % % % % % % % % % % % % % % % % % % % % % % % % % % % % % % % % % % % % % % % % % % % % % % % % % % % % % % % % % % % % % % % % % % % % % % % % % % % % % % % % % % % % % % % % % % % % % % % % % % % % % % % % % % % % % % % % % % % % % % % % % % % % % % % % % % % % % % % % % % % % % % % % % % % % % % %
\section{Further properties of the trajectory attractor}

Let us discuss first some structural properties of the trajectory attractor.

Denote by $Z(h_0):=Z(\mathcal{H}_+(h_0))$ the set of all complete symbols
in $\omega(\mathcal{H}_+(h_0))$.
Recall that a function $\zeta:\mathbb{R}\to V_{div}'$ with $\zeta\in L_{loc}^2(\mathbb{R};V_{div}')$
is a complete symbol in $\omega(\mathcal{H}_+(h_0))$ if $\Pi_+T(t)\zeta\in\omega(\mathcal{H}_+(h_0))$
for all $t\in\mathbb{R}$, where $\Pi_+$ is the restriction operator on the semiaxis $[0,\infty)$.
It can be proved (see \cite[Section 4]{CV} or \cite[Chap. XIV, Section 2]{CV2}) that, due to the strict invariance of $\omega(\mathcal{H}_+(h_0))$, given a symbol $h\in\omega(\mathcal{H}_+(h_0))$
there exists at least one complete symbol $\widehat{h}$ (not necessarily unique) which is an extension of $h$
on $(-\infty,0]$ and such that $\Pi_+T(t)\widehat{h}\in\omega(\mathcal{H}_+(h_0))$
for all $t\in\mathbb{R}$. Note that we have $\Pi_+Z(h_0)=\omega(\mathcal{H}_+(h_0))$.

To every complete symbol $\zeta\in Z(h_0)$ there corresponds by \cite[Chap. XIV, Definition 2.5]{CV2}
(see also \cite[Definition 4.4]{CV}) the kernel $\mathcal{K}_{\zeta}$ in $\mathcal{F}_b$
which consists of the union of all complete trajectories which belong to $\mathcal{F}_b$,
%and are bounded in $\mathcal{F}_b$,
i.e., all weak solutions $z=[v,\psi]:\mathbb{R}\to G_{div}\times H$
with external force $\zeta\in Z(h_0)$ (in the sense of Definition \ref{wsdefn} with $T\in\mathbb{R}$)
satisfying \eqref{ei} on $\mathbb{R}$ (i.e., for all $t\geq s$ and for a.a. $s\in\mathbb{R}$)
that belong to $\mathcal{F}_b$. We recall that the space $(\mathcal{F}_b,d_{\mathcal{F}_b})$ is defined
as the space $(\mathcal{F}_b^+,d_{\mathcal{F}_b}^+)$ with the time interval $(0,\infty)$
replaced by $\mathbb{R}$ in the definitions of $\mathcal{F}_b^+$ and $d_{\mathcal{F}_b^+}$.
The space $(\mathcal{F}_{loc},\Theta_{loc})$ can be defined in the same way.

Set
$$\mathcal{K}_{Z(h_0)}:=\bigcup_{\zeta\in Z(h_0)}\mathcal{K}_{\zeta}.$$
Then, if the assumptions of Theorem \ref{trattex} hold with $F$ bounded in $(-1,1)$
we also have (see, e.g., \cite[Theorem 4.1]{CV})
$$\mathcal{A}_{\mathcal{H}_+(h_0)}=\mathcal{A}_{\omega(\mathcal{H}_+(h_0))}=\Pi_+\mathcal{K}_{Z(h_0)},$$
and the set $\mathcal{K}_{Z(h_0)}$ is compact in $\Theta_{loc}$ and bounded in $\mathcal{F}_b$.

On the other hand, it is not difficult to see that, under the assumptions of Theorem \ref{trattex}, 
$\mathcal{K}_{\zeta}\neq\emptyset$ for all $\zeta\in Z(h_0)$.
Indeed, by virtue of \cite[Theorem 4.1]{CV} (see also \cite[Chap. XIV, Theorem 2.1]{CV2}),
this is a consequence of the fact that the family $\{\mathcal{K}_h^+\}_{h\in\mathcal{H}_+(h_0)}$
of trajectory spaces satisfies the following condition: there exists $R>0$ such that
$B_{\mathcal{F}_b^+}(0,R)\cap\mathcal{K}_h^+\neq\emptyset$ for all $\in\mathcal{H}_+(h_0)$.
In order to check this condition fix an initial datum $z_0^{\ast}=[v_0^{\ast},\psi_0^{\ast}]$,
with $v_0^{\ast},\psi_0^{\ast}$ taken as in Theorem \ref{existence}. We know that for every
$h\in\mathcal{H}_+(h_0)$ there exists a trajectory $z_h^{\ast}\in\mathcal{K}_h^+$ such that
$z_h^{\ast}(0)=z_0^{\ast}$ and such that the energy inequality \eqref{ei} holds
for all $t\geq s$ and for a.a. $s\in(0,\infty)$, {\itshape including $s=0$}. Arguing as in Proposition \ref{dissprop} (cf. \eqref{ta1} written for $s=0$ and all $t\geq 0$) we get an estimate
of the form $d_{\mathcal{F}_b^+}(z_h^{\ast},0)\leq \Lambda(z_0^{\ast},h_0)$ (see also \eqref{tb}),
where the positive constant $\Lambda$ depends on $\mathcal{E}(z_0^{\ast})$ and on the norm
$\|h_0\|_{L^2_{tb}(0,\infty;V_{div}')}$. The above condition is thus fulfilled by choosing $R=\Lambda(z_0^{\ast},h_0)$.

As far as the attraction properties are concerned, we observe that, due to compactness results, 
the trajectory attractor attracts the subsets
of the family $\mathcal{B}_{\mathcal{H}_+(h_0)}$ in some strong topologies. Indeed, setting
\begin{align}
&\mathbb{X}_{\delta_1,\delta_2}:=H^{\delta_1}(\Omega)^d\times H^{\delta_2}(\Omega),\qquad
\mathbb{Y}_{\delta_1,\delta_2}:=H^{-\delta_1}(\Omega)^d\times (H^{\delta_2}(\Omega))',
\end{align}
where $0\leq\delta_1,\delta_2<1$ and using the compact embeddings
\begin{align}
&
L^2(0,M;V_{div}\times V)\cap W^{1,4/3}(0,M;V_{div}'\times V')\hookrightarrow\hookrightarrow
 L^2(0,M;\mathbb{X}_{\delta_1,\delta_2}),\nonumber\\
&L^{\infty}(0,M;G_{div}\times H)\cap W^{1,4/3}(0,M;V_{div}'\times V')\hookrightarrow\hookrightarrow
C([0,M];\mathbb{Y}_{\delta_1,\delta_2}),\nonumber
\end{align}
then Theorem \ref{trattex}
implies the following (see \cite[Chap. XIV, Theorem 2.2]{CV2})
\begin{cor}
Let (A1)-(A7) hold and assume $h_0\in
L^2_{tb}(0,\infty;V_{div}')$. Then, for every $0\leq\delta_1,\delta_2<1$
the trajectory attractor $\mathcal{A}_{\mathcal{H}_+(h_0)}$ from Theorem \ref{trattex}
is compact in $L^2_{loc}([0,\infty);\mathbb{X}_{\delta_1,\delta_2})\cap C([0,\infty);\mathbb{Y}_{\delta_1,\delta_2})$,
bounded in $L^2_{tb}(0,\infty);\mathbb{X}_{\delta_1,\delta_2})\cap C_b([0,\infty);\mathbb{Y}_{\delta_1,\delta_2})$,
and for every
$B\in\mathcal{B}_{\mathcal{H}_+(h_0)}$ and
every $M>0$ we have
\begin{align}
&\mbox{dist}_{L^2(0,M;\mathbb{X}_{\delta_1,\delta_2})}
\Big(\Pi_{[0,M]}T(t)B,\Pi_{[0,M]}\mathcal{A}_{\mathcal{H}_+(h_0)}\Big)\to 0,\nonumber\\
&\mbox{dist}_{C([0,M];\mathbb{Y}_{\delta_1,\delta_2})}
\Big(\Pi_{[0,M]}T(t)B,\Pi_{[0,M]}\mathcal{A}_{\mathcal{H}_+(h_0)}\Big)\to 0,\nonumber
\end{align}
as $t\to+\infty$, where $\mbox{dist}_{X}(A,B)$ denotes the Hausdorff semidistance in the
Banach space $X$ between $A,B\subset X$, and $\Pi_{[0,M]}$ is the restriction operator to the
interval $[0,M]$.
\end{cor}
Let us now define, for every $B\subset\mathcal{K}_{\mathcal{H}_+(h_0)}^+$, the sections
$$B(t):=\Big\{[v(t),\psi(t)]:[v,\psi]\in B \Big\}\subset\mathbb{Y}_{\delta_1,\delta_2},
\qquad t\geq 0.$$
Similarly we set
\begin{align}
&
\mathcal{A}_{\mathcal{H}_+(h_0)}(t):=\Big\{[v(t),\psi(t)]:[v,\psi]\in\mathcal{A}_{\mathcal{H}_+(h_0)} \Big\}\subset
\mathbb{Y}_{\delta_1,\delta_2},
% H^{-\delta_1}(\Omega)^d\times (H^{\delta_2}(\Omega))',
\qquad t\geq 0,\nonumber\\
&
\mathcal{K}_{Z(h_0)}(t):=\Big\{[v(t),\psi(t)]:[v,\psi]\in \mathcal{K}_{Z(h_0)} \Big\}\subset
\mathbb{Y}_{\delta_1,\delta_2},
%H^{-\delta_1}(\Omega)^d\times (H^{\delta_2}(\Omega))',
\qquad t\in\mathbb{R}.\nonumber
\end{align}
Then, as a further consequence of Theorem \ref{trattex} we have
(see \cite[Chap. XIV, Definition 2.6, Corollary 2.2]{CV2}) the following
\begin{cor}
Let (A1)-(A7) hold and assume $h_0\in
L^2_{tb}(0,\infty;V_{div}')$. Then the bounded subset
$$\mathcal{A}_{gl}:=\mathcal{A}_{\mathcal{H}_+(h_0)}(0)=\mathcal{K}_{Z(h_0)}(0)$$
is the uniform (w.r.t. $h\in\mathcal{H}_+(h_0)$) global attractor in $\mathbb{Y}_{\delta_1,\delta_2}$,
%$H^{-\delta_1}(\Omega)^d\times (H^{\delta_2}(\Omega))'$,
$0<\delta_1,\delta_2\leq 1$, of system \eqref{eq1}--\eqref{eq5}, namely
(i) $\mathcal{A}_{gl}$ is compact in $\mathbb{Y}_{\delta_1,\delta_2}$,
%$H^{-\delta_1}(\Omega)^d\times (H^{\delta_2}(\Omega))'$,
(ii) $\mathcal{A}_{gl}$ satisfies the attracting property
$$\mbox{dist}_{\mathbb{Y}_{\delta_1,\delta_2}}(B(t),\mathcal{A}_{gl})\to 0,
\qquad t\to+\infty,$$
for every $B\in\mathcal{B}_{\mathcal{H}_+(h_0)}$, and (iii) $\mathcal{A}_{gl}$
is the minimal set satisfying (i) and (ii).
\end{cor}
\EE

\begin{oss}
\upshape{In the 2D case the energy identity might be exploited
to show the convergence to the trajectory attractor
in the strong topology of the original phase space. 
This was done in \cite{CVZ} for a reaction-diffusion system
without uniqueness.}
\end{oss}

\bigskip

\noindent {\bf Acknowledgments.} This work was partially supported
by the Italian MIUR-PRIN Research Project 2008 ``Transizioni di fase,
isteresi e scale multiple''. The first author was also supported by the
FTP7-IDEAS-ERC-StG Grant $\sharp$200497(BioSMA) and the
FP7-IDEAS-ERC-StG Grant \#256872 (EntroPhase).

\end{document}